\documentclass[11pt]{amsart}

\usepackage[french,english]{babel}
\usepackage{enumitem}
\usepackage[ansinew]{inputenc}
\usepackage{amsmath,amsthm}
\usepackage{hyperref}
\usepackage[T1]{fontenc}
\usepackage{lmodern}
\usepackage{cleveref}
\usepackage{amssymb}
\usepackage{mathtools}
\usepackage{amscd}
\usepackage{marginnote}
\usepackage[all]{xy}
\usepackage{dsfont}
\usepackage{tikz-cd}
\usepackage{mathrsfs}
\usepackage{stmaryrd}

\setenumerate[1]{label=(\alph*)}
\setenumerate[2]{label=(\roman*)}

\theoremstyle{plain}
\newtheorem{thm}{Theorem}[section]
\newtheorem*{thm*}{Theorem}
\newtheorem{lem}[thm]{Lemma}
\newtheorem{cor}[thm]{Corollary}
\newtheorem{prop}[thm]{Proposition}

\theoremstyle{remark}

\theoremstyle{definition}

\newtheorem{defin}[thm]{Definition}

\DeclareFontFamily{OT1}{pzc}{}
\DeclareFontShape{OT1}{pzc}{m}{it}{<-> s * [1.10] pzcmi7t}{}
\DeclareMathAlphabet{\mathpzc}{OT1}{pzc}{m}{it}

\newcommand{\Ecal}{\mathcal{E}}
\newcommand{\Ccal}{\mathcal{C}}
\newcommand{\Fcal}{\mathcal{F}}

\newcommand{\Hcal}{\mathcal{H}}

\newcommand{\Jcal}{\mathcal{J}}
\newcommand{\Lcal}{\mathcal{L}}
\newcommand{\Mcal}{\mathcal{M}}

\newcommand{\Ocal}{\mathcal{O}}
\newcommand{\Pcal}{\mathcal{P}}

\newcommand{\Ucal}{\mathcal{U}}
\newcommand{\Xcal}{\mathcal{X}}
\newcommand{\Ycal}{\mathcal{Y}}

\newcommand{\ZZ}{\mathbb{Z}}
\newcommand{\CC}{\mathbb{C}}

\newcommand{\Zp}{\mathbb{Z}_p}
\newcommand{\Qp}{{\mathbb{Q}_p}}
\newcommand{\con}[1]{\nabla_{#1}}
\newcommand{\Po}{\mathcal{P}} 
\newcommand{\Ed}{E^\vee} 
\newcommand{\Ef}{\hat{E}}
\newcommand{\Etriv}{E^{\mathrm{triv}}}
\newcommand{\Eftriv}{{\hat{E}^{\mathrm{triv}}}}
\newcommand{\Edtriv}{E^{\mathrm{triv,\vee}}}
\newcommand{\Edftriv}{\hat{E}^{\mathrm{triv},\vee}}
\newcommand{\Mtriv}{{M^{\mathrm{triv}}}}
\newcommand{\Gmf}[1]{\widehat{\mathbb{G}}_{m,#1}}
\newcommand{\Gmfabs}{\widehat{\mathbb{G}}_{m}}
\newcommand{\Pof}{\widehat{\Po}}
\newcommand{\id}{\mathrm{id}}
\newcommand{\righteq}{\stackrel{\sim}{\rightarrow}}
\newcommand{\pr}{\mathrm{pr}}

\newcommand{\dd}{\mathrm{d}}
\newcommand{\om}{\underline{\omega}}
\newcommand{\HdR}[2]{\underline{H}^{#1}_{\mathrm{dR}}\left( #2 \right)}
\newcommand{\HdRabs}[2]{H^{#1}_{\mathrm{dR}}\left( #2 \right)}

\newcommand{\EisD}{\prescript{}{D} E^{k,r+1}_{t}}
\newcommand{\VpN}{V\left(\Zp,\Gamma(N)\right)}

\newcommand{\Ln}{\mathcal{L}_n}

\newcommand{\lnD}{l^D_n}

\newcommand{\triv}{\mathrm{triv}}
\newcommand{\Hom}{\mathrm{Hom}} 
\newcommand{\trans}{\tilde{\mathrm{U}}} 

\newcommand{\EispD}{\prescript{}{D}{\mathcal{E}}^{k,r+1}_{(a,b)}}
\newcommand{\Frob}{\mathrm{Frob}}
\newcommand{\muEisD}{\mu^{\mathrm{Eis}}_{D,(a,b)}}

\newcommand{\pthetaD}{\prescript{}{D}\vartheta}
\newcommand{\pthetaDp}{\prescript{}{D}\vartheta^{(p)}}

\newcommand{\VIC}[1]{\mathrm{VIC}\left( #1 \right)}

\newcommand{\gr}{\mathrm{gr}}
\newcommand{\Logn}{\mathrm{Log}^n_{\mathrm{dR}}}

\newcommand{\Log}[1]{\mathrm{Log}^{#1}_{\mathrm{dR}}}
\newcommand{\idM}{\mathds{1}}
 
\newcommand{\polD}{\mathrm{pol}_{D,\mathrm{dR}}}

\newcommand{\LnD}{L_n^D}
\newcommand{\KK}{K}
\newcommand{\scan}{s_{\mathrm{can}}}

\newcommand{\Ee}{\mathscr{E}}

\newcommand{\Mm}{\mathscr{M}}
\newcommand{\Ss}{\mathscr{S}}
\newcommand{\Uu}{\mathscr{U}}
\newcommand{\Vv}{\mathscr{V}}
\newcommand{\Xx}{\mathscr{X}}
\newcommand{\Yy}{\mathscr{Y}}
\newcommand{\DR}{\mathrm{DR}}
\newcommand{\an}{{an}}
\newcommand{\rig}{{rig}}
\newcommand{\Mftriv}{\Mcal^{\mathrm{triv}}}
\newcommand{\Hsyn}[2]{\underline{H}^{#1}_{\mathrm{syn}}\left( #2 \right)}
\newcommand{\Hrigabs}[2]{H^{#1}_{\mathrm{rig}}\left( #2 \right)}
\newcommand{\Hsynabs}[2]{H^{#1}_{\mathrm{syn}}\left( #2 \right)}
\newcommand{\Logsyn}[1]{\mathpzc{Log}^{#1}_{\mathrm{syn}}}
\newcommand{\LogsyndR}[1]{\mathrm{Log}^{#1}_{\mathrm{dR}}}
\newcommand{\polsyn}{\mathpzc{pol}_{D,\mathrm{syn}}}
\newcommand{\ord}{\mathrm{ord}}
\newcommand{\eD}{\hat{\mathrm{e}}_{{\tilde{t}},(k,l)}}
\newcommand{\Cp}{{\mathbb{C}_p}}

\DeclareMathOperator{\Spf}{{\mathrm{Spf}}}
\DeclareMathOperator{\Spec}{Spec}
\DeclareMathOperator{\Res}{Res}
\DeclareMathOperator{\Sym}{\underline{Sym}}
\DeclareMathOperator{\Inf}{Inf}
\DeclareMathOperator{\TSym}{\underline{TSym}}

\DeclareMathOperator{\Pic}{\underline{\mathrm{Pic}}^0_{E/S}}

\DeclareMathOperator{\Ext}{\mathrm{Ext}}
\DeclareMathOperator{\Meas}{\mathrm{Meas}}

\title[The syntomic realization of the elliptic polylogarithm]{The syntomic realization of the elliptic polylogarithm via the Poincar\'e bundle}

\author{Johannes Sprang}
\email{johannes.sprang@mathematik.uni-regensburg.de }
\date{}

\hypersetup{
pdfauthor = {J. Sprang},
pdftitle = {The syntomic realization of the elliptic polylogarithm via the Poincar\'e bundle}
}

\begin{document}

\begin{abstract}
We give an explicit description of the syntomic elliptic polylogarithm on the universal elliptic curve over the ordinary locus of the modular curve in terms of certain $p$-adic analytic moment functions associated to Katz' two-variable $p$-adic Eisenstein measure. The present work generalizes previous results of Bannai--Kobayashi--Tsuji and Bannai--Kings on the syntomic Eisenstein classes.
\end{abstract}

\maketitle

\section{Introduction}
In his seminal paper \cite{beilinsonConj}, Beilinson has formulated his well-known conjectures expressing values of $L$-functions up to a rational factor in terms of motivic cohomology classes under the image of the regulator map to Deligne cohomology. In order to study particular cases of these conjectures, it is indispensable to construct cohomology classes in Deligne cohomology of motivic origin and then relate them to $L$-values. The elliptic poylogarithm, defined by Beilinson and Levin in \cite{beilinson_levin}, is a very important source of such cohomology classes. For a $p$-adic variant of the Beilinson conjectures one needs a suitable $p$-adic equivalent of Deligne cohomology. It turns out that syntomic cohomology can be seen as a $p$-adic substitute for Deligne cohomology. In the case of good reduction Bannai has proven that syntomic cohomology can be seen as absolute $p$-adic Hodge cohomology \cite{bannai_absHodge}. Thus, for studying particular cases of $p$-adic versions of the Beilinson conjectures, like the Perrin-Riou conjecture, we wish to have a good understanding of polylogarithmic cohomology classes in syntomic cohomology.\par 
The syntomic Eisenstein classes are the cohomology classes obtained by restricting the elliptic polylogarithm along torsion sections. This describes only a shadow of the elliptic polylogarithm, but the syntomic Eisenstein classes and related cohomology classes play an important role in recent research: Bertolini, Darmon and Rotger initiated a program for systematically studying Rankin-Selberg convolutions in $p$-adic families \cite{BDR1}, \cite{BDR2}. This has been continued by work of Kings, Loeffler and Zerbes. Eisenstein classes are the key in the construction of the Rankin-Eisenstein classes considered by Kings--Loeffler--Zerbes in \cite{KLZ1}. They have related the syntomic Rankin-Eisenstein classes to $p$-adic Rankin $L$-functions. Relating the etale Rankin-Eisenstein classes to the syntomic ones allows them to prove an explicit reciprocity law proving the non-triviality of the Euler system of Rankin-Eisenstein classes \cite{KLZ2}.\par 
 While we have a good understanding of the Deligne realization of the elliptic polylogarithm from Beilinson--Levin, there are only two partial results on its syntomic realization: In \cite{BKT} Bannai--Kobayashi--Tsuji have given an explicit description of the syntomic realization for a single elliptic curve with complex multiplication defined over a subfield of $\CC$. Unfortunately, this method does not immediately generalize to more general elliptic curves, since it builds on explicit calculations involving the analytically defined Kronecker theta function. Complex multiplication is needed to guarantee the algebraicity of the involved theta function. On the other hand, according to work of Bannai and Kings, we have a good understanding of the syntomic Eisenstein classes on the ordinary locus of the modular curve in terms of $p$-adic Eisenstein series. 
While the result of Bannai--Kobayashi--Tsuji only covers the case of single CM elliptic curves, the result of Bannai--Kings is limited to the Eisenstein classes on the ordinary locus of the modular curve. In this paper, we give a common generalization to both results: We give an explicit description of the syntomic elliptic polylogarithm on the universal elliptic curve over the ordinary locus of the modular curve in terms of certain $p$-adic analytic moment functions associated to Katz' two-variable $p$-adic Eisenstein measure. 

\begin{thm*}[cf.~Theorem\footnote{ for a more precise version of the theorem, see the main body of the text} \ref{thm_mainThm}]
	There is a compatible system of overconvergent sections in the syntomic logarithm sheaves $\rho_n\in\Gamma\left(\bar{\Ee}_K,j_D^\dagger(\LogsyndR{n})\right)$ describing the $D$-variant of the syntomic polylogarithm on the ordinary locus of the modular curve 
	\[
		\polsyn=\left([\rho_n]\right)_{n\geq 0}\in\varprojlim_n \Hsynabs{1}{\Uu_D,\Logsyn{n}(1)}.
	\]
	In tubular neighbourhoods $]\tilde{t}[$ of torsion sections we have the following explicit description of these overconvergent sections in terms of moment functions of Katz' two-variable $p$-adic Eisenstein measure
	\[
		\rho_n|_{]\tilde{t}[}(s)=\sum_{k+l\leq n} (-1)^{l} l! \int_{\Zp^\times\times \Zp} y^k x^{-(l+1)} (1+s)^x \dd \mu^{\mathrm{Eis}}_{D,(a,b)}(x,y)\cdot \hat{\omega}^{[k,l]}
	\]
	with values in $p$-adic modular forms.
\end{thm*}
Although the complex analytic Kronecker theta function is not applicable in our setup, we essentially follow the strategy of Bannai--Kobayashi--Tsuji. The geometry of the Poincar\'e bundle serves as a substitute for the Kronecker theta function. More precisely, we make use of the purely algebraically defined \emph{Kronecker section} of the Poincar\'e bundle, which has been fruitfully applied in \cite{EisensteinPoincare} to study algebraic and $p$-adic properties of Eisenstein--Kronecker series. Since the syntomic realization refines the algebraic de Rham realization, we need a good understanding of the latter. This has been established in our earlier work \cite{deRham}. Building on previous work of Scheider, we have given there an explicit description of the algebraic de Rham realization for arbitrary families of elliptic curves via the Poincar\'e bundle. It is again the Poincar\'e bundle which allows us to relate the syntomic realization to Katz' two variable $p$-adic Eisenstein measure: Here we build on \cite{EisensteinPoincare}, where we have constructed Katz' Eisenstein measure using $p$-adic theta functions associated to the Poincar\'e bundle.\par 

\section*{Acknowledgement}
The results presented in this paper are part of my Ph.D. thesis \cite{PhD}. It is a pleasure to thank my advisor Guido Kings for his guidance during the last years. Further, I would like to thank Shinichi Kobayashi for all the valuable suggestions on my PhD thesis. The author would also like to thank the collaborative research centre SFB 1085 ``Higher Invariants'' by the Deutsche Forschungsgemeinschaft for its support.

\section{Rigid syntomic cohomology}
Syntomic cohomology can be seen as the $p$-adic analogue of Deligne--Beilinson cohomology. Indeed, in the case of good reduction Bannai has proven that syntomic cohomology can be seen as absolute $p$-adic Hodge cohomology \cite{bannai_absHodge}. The work of Deglise and Nizio\l{} generalizes this to arbitrary smooth proper schemes over a discretely valued field of mixed characteristic\cite{deglise_niziol}. The approach of Deglise--Nizio\l{} allows further the construction of a ring spectrum in the motivic homotopy category of Morel--Voevodsky representing syntomic cohomology. In their approach coefficients for syntomic cohomology can be defined abstractly as modules over this ring spectrum. Nevertheless, we will use rigid syntomic cohomology as developed by Bannai for describing the syntomic realization of the elliptic polylogarithm. Indeed, since we want an explicit description of the polylogarithm class, we need explicit complexes computing syntomic cohomology.\par 
In this section we briefly recall the definition and basic properties of rigid syntomic cohomology. We follow closely the appendix of \cite{bannai_kings}. In particular, we use their modification of the definition of smooth pair allowing overconvergent Frobenii which are not globally defined. Let $K/\Qp$ be a finite unramified extension with ring of integers $\Ocal_K$, residue field $k$ and Frobenius morphism $\sigma:K\rightarrow K$. \par 
A \emph{smooth pair} is a tuple $\Xx=(X,\bar{X})$ consisting of a smooth scheme $X$ of finite type over $\Ocal_K$ together with a smooth compactification $\bar{X}$ of $X$ with complement $D:=\bar{X}\setminus X$ a simple normal crossing divisor relative $\Spec \Ocal_K$. We denote the formal completion of $X$ w.r.t $X_k:=X\times_{\Spec\Ocal_K}\Spec k$ by $\Xcal$ and the formal completion of $\bar{X}$ w.r.t $\bar{X}_k$ by $\bar{\Xcal}$. The rigid analytic spaces associated with $\Xcal$ resp. $\bar{\Xcal}$ will be denoted by $\Xcal_K$ resp. $\bar{\Xcal}_K$. An \emph{overconvergent Frobenius} $\phi_X=(\phi,\phi_V)$ on a smooth pair $\Xx=(X,\bar{X})$ consists of: A morphism of $\Ocal_K$-formal schemes
	\[
		\phi: \Xcal\rightarrow \Xcal
	\]
	lifting the absolute Frobenius on $X_k$ and an extension of $\phi$ to a morphism of rigid analytic spaces
	\[
		\phi_V: V\rightarrow \bar{\Xcal}_K
	\]
	to some strict neighbourhood $V$ of $\Xcal_K$ in $\bar{\Xcal}_K$. A smooth pair together with an overconvergent Frobenius $\Xx=(X,\bar{X},\phi,\phi_V)$ will be called \emph{syntomic datum}.\par

For a smooth pair $\Xx=(X,\bar{X})$ let us write $X_K$ and $\bar{X}_K$ for the generic fibers and $X_K^\an$ resp. $\bar{X}^\an$ for the associated rigid analytic spaces. Then, $X_K^\an$ is a strict neighbourhood of $j:\Xcal_K\hookrightarrow \bar{\Xcal}_K$. A coherent module $M$ on $\bar{X}_K$ with integrable connection
\[
	\nabla:M\rightarrow M\otimes\Omega^1_{\bar{X}_K}(\log D)
\]
and logarithmic poles along $D$ induces an overconvergent connection $(M^\rig,\nabla^\rig)$ on $M^\rig:=j^\dagger (M|_{X_K^\an})$. The category of filtered overconvergent $F$-isocrystals on $\Xx$ serves as coefficients for rigid syntomic cohomology and may be realized as follows.
\begin{defin}
	Let the category $S(\Xx)$ of \emph{filtered overconvergent $F$-isocrystals} on $\Xx=(X,\bar{X})$ be the category consisting of $4$-tuples
	\[
		\Mcal=(M,\nabla,F^\bullet,\Phi_M)
	\]
	with: $M$ a coherent $\Ocal_{\bar{X}_K}$-module with integrable connection
	\[
	\nabla:M\rightarrow M\otimes\Omega^1_{\bar{X}_K}(\log D)
\]
with logarithmic poles along $D=\bar{X}_K\setminus X_K$. $F^\bullet$ a descending exhaustive and separating filtration on $M$ satisfying Griffith transversality:
	\[
		\nabla(F^\bullet M)\subseteq F^{\bullet-1}(M)\otimes \Omega^1_{\bar{X}_K}(\log D)
	\] 
	And a horizontal isomorphism
	\[
		\Phi_M: F_\sigma^* M^\rig\rightarrow M^\rig.
	\]
	where $F_\sigma$ is the Frobenius endofunctor on the category of overconvergent isocrystals defined in \cite{berthelot}. $\Phi_M$ will be called a \emph{Frobenius structure}. Morphisms in this category are morphisms of $\Ocal_{\bar{X}_K}$-modules respecting the additional structure.
\end{defin}

If one has a fixed overconvergent Frobenius on the smooth pair $\Xx=(\bar{X},X)$, one can realize a Frobenius structure more concretely as a horizontal isomorphism
\[
	\phi_V^*M^\rig\rightarrow M^\rig.
\] 
A morphism of pairs $\Xx=(X,\bar{x})\rightarrow \Yy=(Y,\bar{Y})$ is a morphism $f:\bar{X}\rightarrow \bar{Y}$ such that $f(X)\subseteq Y$. A morphism of pairs is called smooth, proper, etc, if $f|_X$ is smooth, proper, etc. For smooth morphisms of smooth pairs we define the higher direct image of filtered overconvergent $F$-isocrystals as follows. Let $D':=\bar{Y}\setminus Y$. The sheaf of relative logarithmic differentials is defined as the cokernel in the following short exact sequence:
\[
	\begin{tikzcd}
		0\ar[r] & f^*\Omega^1_{\bar{Y}_K}(\log D')\ar[r] & \Omega^1_{\bar{X}_K}(\log D) \ar[r] & \Omega^1_{\bar{X}_K/\bar{Y}_K,\log}\ar[r] & 0
	\end{tikzcd}
\]
and $\Omega^p_{\bar{X}_K/\bar{Y}_K,\log}:=\Lambda^p \Omega^1_{\bar{X}_K/\bar{Y}_K,\log}$. For $\Mcal=(M,\nabla,F^\bullet,\Phi_M)\in S(\Xx)$ we can define the following algebraic and rigid relative de Rham complexes
\[
	\DR^\bullet_{X/Y}(M):=M\otimes_{\Ocal_{\bar{X}}}\Omega^\bullet_{\bar{X}_K/\bar{Y}_K,\log}
\]
and
\[
	\DR^\bullet_{X/Y}(M^\rig):=M^\rig \otimes_{j^\dagger \Ocal_{\bar{\Xcal}_K}}j^\dagger \Omega^\bullet_{\bar{\Xcal}_K/\bar{\Ycal}_K}
\]
and their higher direct images
\[
	R^p f_* \DR^\bullet_{X/Y}(M),\quad R^p f_{\rig,*} \DR^\bullet_{X/Y}(M^\rig).
\]
In the special case $\Xx\xrightarrow{f}\Vv:=(\Ocal_K,\Ocal_K)$ both
\[
	\HdRabs{p}{X_K,M}:=R^p f_* \DR^\bullet_{X/K}(M),\quad \Hrigabs{p}{X_k,M^\rig}:=R^p f_{\rig,*} \DR^\bullet_{X/K}(M^\rig)
\]
are $K$-vector spaces.\par 
While $R^p f_* \DR^\bullet_{X/Y}(M)$ is equipped with the Hodge-Filtration $F^\bullet$ and the Gauss--Manin connection $\con{\mathrm{GM}}$, the rigid cohomology $R f_{\rig,*} \DR^\bullet_{X/Y}(M^\rig)$ is equipped with a Frobenius structure $\Phi$. If we write $j_Y:\Ycal_K\hookrightarrow \bar{\Ycal}_K$ for the inclusion, we have a comparison map
\[
	\Theta_{\Xx/\Yy}: j_Y^\dagger\left( R^p f_* \DR^\bullet_{X/Y}(M)|_{Y_K^\an} \right)\rightarrow R^p f_{\rig,*} \DR^\bullet_{X/Y}(M^\rig).
\]
Thus, whenever $\Theta_{\Xx/\Yy}$ is an isomorphism, we obtain a structure of a filtered overconvergent $F$-isocrystal over $\Yy$:
\[
	\Hsyn{p}{\Xx/\Yy,\Mcal}:=\left(R f_* \DR^\bullet_{X/Y}(M),\con{\mathrm{GM}},F^\bullet,\Phi  \right)\in S(\Yy).
\]
It is known that for proper maps $\pi:\Xx\rightarrow \Yy$ the comparison map $\Theta_{\Xx/\Yy}$ is always an isomorphism \cite[Prop. A.7.]{bannai_kings}.
\begin{defin}
	A filtered overconvergent $F$-isocrystal $\Mcal=(M,\nabla,F^\bullet,\Phi_M)\in S(\Xx)$ is called \emph{admissible} if:
	\begin{enumerate}
	\item The Hodge to de Rham spectral sequence
	\[
		E_1^{p,q}=H^p(\bar{X}_K,\gr^q_F \DR^\bullet_{X/K}(M))\implies \HdRabs{p+q}{X_K,M}
	\]
	degenerates at $E_1$.
	\item $\Theta_{\Xx/(\Ocal_K,\Ocal_K)}: \HdRabs{p}{X_K,M}\rightarrow \Hrigabs{p}{X_k,M^\rig}$
	is an isomorphism.
	\item The $K$-vector space $\HdRabs{p}{X_K,M}\righteq \Hrigabs{p}{X_k,M^\rig}$ with Hodge filtration coming from $H^{p}_{\mathrm{dR}}$ and Frobenius structure coming from $H^{p}_{\mathrm{rig}}$ is weakly admissible in the sense of Fontaine.
	\end{enumerate}
	Let us write $S(\Xx)^{adm}$ for the full subcategory of admissible objects.
\end{defin}
We will also need the following relative version of `admissible' from \cite[Def. 5.8.12]{solomon}:
\begin{defin}
	Let $\pi:\Xx\rightarrow \Yy$ be a smooth morphism of smooth pairs. A filtered overconvergent $F$-isocrystal $\Mcal=(M,\nabla,F^\bullet,\Phi_M)\in S(\Xx)$ is called \emph{$\pi$-admissible} if:
	\begin{enumerate}
		\item $\Theta_{\Xcal/\Ycal}$ is an isomorphism.
		\item The obtained filtered overconvergent $F$-isocrystals over $\Yy$
		\[
			\Hsyn{p}{\Xx/\Yy,\Mcal}:=\left(R^p f_* \DR^\bullet_{X/Y}(M),\con{\mathrm{GM}},F^\bullet,\Phi  \right)\in S(\Yy).
		\]
		are admissible.
	\end{enumerate}
	Let us write $S(\Xx)^{\pi-adm}$ for the full subcategory of $\pi$-admissible objects.
\end{defin}
For $\pi:\Xx\rightarrow \Yy$ a smooth morphism of smooth pairs we obtain functors
\[
	\Hsyn{p}{\Xx/\Yy,\cdot}: S(\Xx)^{\pi-adm} \rightarrow S(\Yy)^{adm}.
\]
Let us briefly recall the definition of rigid-syntomic cohomology as given by Bannai. We follow the exposition in \cite{bannai_kings}: Let $\Xx=(X,\bar{X},\phi,\phi_V)$ be a syntomic datum and $\Mcal=(M,\nabla,F^\bullet,\Phi_M)$ be a filtered overconvergent $F$-isocrystal. For a finite Zariski covering $\mathfrak{U}=(\bar{U}_i)_{i\in I}$ of $\bar{X}$ set $\bar{U}_{i_0,...,i_n,K}:=\bigcap_{0\leq j\leq n}\bar{U}_{i_j,K}$. $\mathfrak{U}$ induces a covering $(\Ucal_{i,K})_{i\in I}$ of $\Xcal_K$ obtained via the completion of $U_i\cap X$ along its special fiber. Let us write
\[
	j_{i_0,...,i_n}: \Ucal_{i_0,...,i_n,K}:=\bigcap_{0\leq j\leq n} \Ucal_{i_j,K} \hookrightarrow \bar{\Xcal}_K
\]
for the inclusion. The total complex associated with the \v{C}ech complex 
\[
	\prod_{i} \Gamma\left(\bar{U}_{i,K}, \DR_{\mathrm{dR}}^\bullet(M) \right)\rightarrow \prod_{i_0,i_1} \Gamma\left(\bar{U}_{i_0,i_1,K}, \DR_{\mathrm{dR}}^\bullet(M) \right)\rightarrow ...
\]
will be denoted by $R^\bullet_{\mathrm{dR}}(\mathfrak{U},\Mcal)$. Similarly, let us define $R^\bullet_{\mathrm{rig}}(\mathfrak{U},\Mcal)$ as the total complex associated with:
\[
	\prod_{i} \Gamma\left(\bar{\Xcal}_{K}, j^\dagger_{i}\DR_{\rig}^\bullet(M^\rig) \right)\rightarrow \prod_{i_0,i_1} \Gamma\left(\bar{\Xcal}_{K}, j^\dagger_{i_0,i_1}\DR_{\rig}^\bullet(M^\rig) \right)\rightarrow ...
\]
The Frobenius structure $\Phi_M$ together with the overconvergent Frobenius $\phi_X=(\phi,\phi_V)$ induce
\[
	\phi_{\mathfrak{U}}:K\otimes_{\sigma,K} R^\bullet_{\mathrm{rig}}(\mathfrak{U},\Mcal)\rightarrow R^\bullet_{\mathrm{rig}}(\mathfrak{U},\Mcal)
\]
and the comparison map $\Theta_{X/K}$ induces
\[
	\Theta_{\mathfrak{U}}: R^\bullet_{\mathrm{dR}}(\mathfrak{U},\Mcal)\rightarrow  R^\bullet_{\mathrm{rig}}(\mathfrak{U},\Mcal).
\]
Let
\[
	R^\bullet_{\mathrm{syn}}(\mathfrak{U},\Mcal):=\mathrm{Cone}\left( F^0R^\bullet_{\mathrm{dR}}(\mathfrak{U},\Mcal)\xrightarrow[]{(1-\phi_{\mathfrak{U}})\circ \Theta_{\mathfrak{U}}}  R^\bullet_{\mathrm{rig}}(\mathfrak{U},\Mcal) \right)[1]
\]
where $F^\bullet$ is the filtration induced by the Hodge filtration.
\begin{defin}
	The rigid syntomic cohomology of $\Xx$ with coefficients in $\Mcal$ is defined by
	\[
		\Hsynabs{n}{\Xx,\Mcal}:=\varinjlim_{\mathfrak{U}} H^n\left( R^\bullet_{\mathrm{syn}}(\mathfrak{U},\Mcal) \right)
	\]
	where the limit is taken over all finite Zariski coverings.
\end{defin}
By its very definition we have a long exact sequence
\begin{equation}\label{eq_Syn_seq}
\begin{tikzcd}
	..\ar[r] & F^0 \HdRabs{m}{X_K,M} \ar[r,"1-\phi"] & \Hrigabs{m}{X_k,M^\rig}\ar[r] & \Hsynabs{m+1}{\Xx,\Mcal}\ar[r] & ...
\end{tikzcd}
\end{equation}
Above we have defined functors
\[
	\Hsyn{p}{\Xx/\Yy,\cdot}: S(\Xx)^{\pi-adm} \rightarrow S(\Yy)^{adm}.
\]
The reason for the chosen notation for this functor is the following spectral sequence \cite[Theorem 5.9.1]{solomon}.
For $\Mcal=(M,\nabla,F^\bullet,\Phi_M)\in S(\Xx)^{\pi-adm}$ and $\pi:\Xx\rightarrow\Yy $ a smooth morphism of smooth pairs there is a Leray spectral sequence:
\[
	E^{p,q}_2=\Hsynabs{p}{\Yy,\Hsyn{p}{\Xx/\Yy,\Mcal}}\implies E^{p+q}=\Hsynabs{p+q}{\Xx,\Mcal}.
\]
Either by this spectral sequence or directly by the above long exact sequence, we deduce the following:
\begin{cor}\label{SR_Syn_corexseq}
	For $\mathscr{V}:=(\Ocal_K,\Ocal_K)$ we have the short exact sequence:
	\[
	 0\rightarrow \Hsynabs{1}{\mathscr{V},\Hsyn{m}{\Xx,\Mcal}} \rightarrow \Hsynabs{m+1}{\Xx,\Mcal} \rightarrow \Hsynabs{0}{\mathscr{V},\Hsyn{m+1}{\Xx,\Mcal}} \rightarrow 0
	\]
\end{cor}
\begin{defin}
	The \emph{boundary map}
	\[
		\delta:\Hsynabs{m}{\Xx,\Mcal}\rightarrow \HdRabs{m}{X_K,M}
	\]
	is defined as the composition
	\[
		\Hsynabs{m}{\Xx,\Mcal}\rightarrow \Hsynabs{0}{\mathscr{V},\Hsyn{m}{\Xx,\Mcal}}
	\]
	with the inclusion
	\[
		\Hsynabs{0}{\mathscr{V},\Hsyn{m}{\Xx,\Mcal}}=\ker\left( F^0\HdRabs{m}{X_K,M}\xrightarrow[]{1-\phi}\Hrigabs{m}{X_k,M^\rig} \right)\subseteq \HdRabs{m}{X_K,M}
	\]
\end{defin}
In general, the category of filtered overconvergent $F$-isocrystals $S(\Xx)$ is not Abelian. As in \cite[Rem 1.15]{bannai} we will regard the category $S(\Xx)$ as an exact category with the class of exact sequences given by
\[
	0\rightarrow M'\rightarrow M\rightarrow M''\rightarrow 0
\]
such that the underlying sequence of $\Ocal_{\bar{X}_K}$-modules is exact and the morphisms in the sequence are strictly compatible with the filtrations. The Tate objects $K(n)\in S(\Xx)$ are defined as
\[
	K(n)=(\Ocal_{\bar{X}_K},\dd\colon \mathcal{O}_{\bar{X}_K}\rightarrow \Omega^1_{\bar{X}_K},F^\bullet,\Phi)
\]
with $F^{-j}\Ocal_{\bar{X}_K}=\Ocal_{\bar{X}_K}\subseteq F^{-j+1}\Ocal_{\bar{X}_K}=0$ and $\Phi(1)=p^{-j}$. Let us write $\VIC{X_K/K}$ for the category of vector bundles on $K_K$ with an integrable $K$-connection.
\begin{prop}[{\cite[Proposition 4.4]{bannai}}]\label{SR_propExt}
For $i=0,1$ there is a canonical isomorphism
\[
	\Ext_{S(\Xx)}^i(K(0),\Mcal)\righteq \Hsynabs{i}{\Xx,\Mcal}
\]
fitting into the commutative diagram
\[
	\begin{tikzcd}
		\Ext_{S(\Xx)}^i(K(0),\Mcal) \ar[r,"\nu"]\ar[d,"\cong"] & \Ext_{\VIC{X_K/K}}^i(K(0),M)\ar[d,"\cong"]\\
		\Hsynabs{i}{\Xx,\Mcal}\ar[r,"\delta"] & \HdRabs{i}{X_K,M}
	\end{tikzcd}
\]
where $\nu$ is the map forgetting the Hodge filtration and the Frobenius structure.
\end{prop}
For an admissible filtered overconvergent $F$-isocrystal $\Mcal=(M,\nabla,F^\bullet,\Phi_M)$ on $(X,\bar{X})$ let us write
\[
	\Phi\colon \Gamma(\bar{\Xcal}_K,M^\rig)\rightarrow \Gamma(\bar{\Xcal}_K,M^\rig), \quad \alpha \mapsto \Phi_M(\phi^*\alpha)
\]
for the map induced by the Frobenius structure. Last but not least, let us recall the following useful description of classes in $\Hsynabs{1}{\Xx,\Mcal}$ if $F^0M=0$:
\begin{prop}[{\cite[Proposition A.16]{bannai_kings}}]\label{prop_BK}
 Let $\Mcal=(M,\nabla,F,\Phi)$ be an admissible filtered overconvergent $F$-isocrystal with $F^0M=0$. A cohomology class
  \[
  	[\alpha]\in\Hsynabs{1}{\Xx,\Mcal}
  \]
  is given uniquely by a pair $(\alpha,\xi)$ with
  \[
  	\alpha\in\Gamma(\bar{\Xcal}_K,M^\rig),\quad \xi\in \Gamma(\bar{X}_K, F^{-1}M\otimes \Omega^1_{\bar{X}_K}(\log D))
  \]
  satisfying the conditions:
  \[
  	\nabla(\alpha)=(1-\Phi)(\xi),\quad \nabla(\xi)=0
  \]
\end{prop}
In particular, this result will apply to the polylogarithm class. Indeed, we will see that the differential equation of overconvergent functions describing the rigid syntomic polylogarithm class is just a restatement of the abstract differential equation
\[
  	\nabla(\alpha)=(1-\Phi)(\xi).
 \]
 \begin{cor}[{\cite[Corollary A.17]{bannai_kings}\label{SR_corDRclass}}]
 	Suppose $(\alpha,\xi)=[\alpha,\xi]\in \Hsynabs{1}{\Xx,\Mcal}$ is as in the previous proposition. Then, the image of $[\alpha,\xi]$ under
 	\[
 		\Hsynabs{1}{\Xx,\Mcal}\rightarrow \HdRabs{1}{\Xx,M}
 	\]
 	is given by $[\xi]$.
 \end{cor}

\section{The syntomic polylogarithm class}
The syntomic elliptic polylogarithm is a pro-system of cohomology classes in syntomic cohomology with coefficients in certain filtered overconvergent $F$-isocrystals,  \emph{called the syntomic logarithm sheaves}. The aim of this section is to introduce the syntomic logarithm sheaves and to define the syntomic polylogarithm class. Let us start with the definition of the syntomic logarithm sheaves. As before, let $K/\Qp$ be a finite unramified extension. Let $\pi:\Ee=(E,\bar{E},\phi_E)\rightarrow \Ss=(S,\bar{S},\phi_S)$ be a morphism of syntomic data with $\pi:E\rightarrow S$ an elliptic curve over some affine scheme $S$.\par
Since $\pi$ is proper, $\Hcal:=\Hsyn{1}{\Ee/\Ss,K(1)}\in S(\Ss)^{ad}$ and $\Hcal^\vee\in S(\Ss)^{ad}$ are well defined admissible filtered overconvergent $F$-isocrystals. Let us define $\Hcal_E:=\pi^*\Hcal$. Applying the Leray spectral sequence for syntomic cohomology gives a split short exact sequence:
\[
	\begin{tikzcd}
		0\ar[r] & \Ext^1_{S(\Ss)}(K(0),\Hcal)\ar[r,"\pi^*"] & \Ext^1_{S(\Ee)}(K(0),\Hcal_E)\ar[l,swap,bend right,"e^*"] \ar[r,"\tilde{\delta}"] & \Hom_{S(\Ss)}(\Hcal,\Hcal)\ar[r] & 0
	\end{tikzcd}
\]
Let $[\Logsyn{1}]\in \Ext^1_{S(\Ee)}(K(0),\Hcal_E)$ be the unique extension class satisfying $e^*[\Logsyn{1}]=0$ and $\tilde{\delta}[\Logsyn{1}]=\id_{\Hcal}$. A representative of this extension class does not have any non-trivial automorphisms, thus it is uniquely determined by its extension class up to unique isomorphism. An overconvergent $F$-isocrystal 
\[
	\Logsyn{1}=(\LogsyndR{1},\nabla,F^\bullet_{\Logsyn{}},\Phi_{\Logsyn{}})
\]
representing the extension class $[\Logsyn{1}]$ will be called \emph{the first syntomic logarithm sheaf}. The higher logarithm sheaves are defined by taking tensor symmetric powers $$\Logsyn{n}:=\TSym^n\Logsyn{1}.$$ Here, recall that the $n$-th tensor symmetric power $\TSym^n \LogsyndR{1}$ is defined by
\[
	\TSym^n \LogsyndR{1}:= \left( \LogsyndR{1}\otimes...\otimes \LogsyndR{1}\right)^{S_n}
\]
where the symmetric group $S_n$ acts by permutation. The connection, the filtration and the Frobenius structure on $\LogsyndR{1}$ induce an overconvergent $F$-isocrystal structure on $\LogsyndR{n}:=\TSym^n \LogsyndR{1}$.

The canonical map $\Logsyn{1} \twoheadrightarrow K(0)$ induces transition maps $\Logsyn{n+1}\twoheadrightarrow \Logsyn{n}$. Since the extension class of the first logarithm sheaf pulls back to zero along $e$,  we get a canonical isomorphism
\[
	e^*\Logsyn{1}\righteq K(0)\oplus \Hcal.
\]
By passing to tensor symmetric powers, we obtain
\[
	e^*\Logsyn{n}\righteq \bigoplus_{k=0}^n\TSym^k \Hcal.
\]

In particular, $1\in K(0)=\TSym^0 \Hcal$ gives us a canonical horizontal section $\idM^{(n)}$ in $e^*\Logsyn{n}$. Although the section $\idM^{(n)}\in\Gamma(S,\LogsyndR{1})$ is uniquely determined by the datum of the logarithm sheaves it will be convenient to include it into the structure defining the $n$-th syntomic logarithm sheaf. So sometimes we will write
\[
	\Logsyn{n}=(\LogsyndR{n},\nabla,F^\bullet_{\Logsyn{}},\Phi_{\Logsyn{}},\idM^{(n)}).
\]
\par
Since $\pi$ is proper, we have
\[
	\Hsyn{i}{\Ee/\Ss,\Logsyn{n}}\in S(\Ss).
\]
One can compute the syntomic cohomology of the logarithm sheaves along the same lines as in the de Rham realization \cite[\S 1.2]{rene}:
\begin{prop}\label{SR_Log_propCohom} Let $\pi:\Ee\rightarrow\Ss$ be as before.
The transition maps 
\[
	\Hsyn{i}{\Ee/\Ss,\Logsyn{n+1}}\rightarrow \Hsyn{i}{\Ee/\Ss,\Logsyn{n}}
\]
are zero for $i=0,1$ and are isomorphisms for $i=2$. In particular, the trace isomorphism for $i=2$ gives canonical isomorphisms
\[
	\Hsyn{2}{\Ee/\Ss,\Logsyn{n}}\righteq ... \righteq \Hsyn{2}{\Ee/\Ss,\Logsyn{0}}\righteq K(0)
\]
\end{prop}
\begin{proof}
Since the underlying $F$-isocrystals are the relative de Rham cohomology sheaves plus extra structure, the result follows from the corresponding result for de Rham cohomology \cite[Proposition 2.2]{deRham}.
\end{proof}
For the definition of the ($D$-variant) of the elliptic polylogarithm in rigid syntomic cohomology consider the following diagram of smooth pairs: For $D>1$ define $U_D:=E\setminus E[D]$. 
\[
\begin{tikzcd}
	\Uu_D:=(U_D,\bar{E})\ar[r,hook,"j_D"]\ar[dr,swap,"\pi_{\Uu_D}"] & \Ee:=(E,\bar{E})\ar[d,"\pi_\Ee"] & \Ee[D]:=(E[D],\overline{E[D]})\ar[l,swap,hook,"i_D"]\ar[dl,"\pi_{\Ee[D]}"]\\
	& \Ss=(M,\bar{M}) &
\end{tikzcd}
\]
Combining the localization sequence with the vanishing results from \cref{SR_Log_propCohom} gives the exact sequence:
	\[
	\begin{tikzcd}
		0 \ar[r] & \varprojlim_n \Hsynabs{1}{\Uu_D,\Logsyn{n}(1)}\ar[r,"\Res"]  & \varprojlim_n \Hsynabs{0}{\Ee[D],\Logsyn{n}(1)|_{\Ee[D]}}\ar[r,"\mathrm{aug}"] & K.
	\end{tikzcd}
	\]
Let us write $1_e\in \Gamma(E[D],\Ocal_{E[D]})$ for the horizontal section supported on $[e]$ and $1_{E[D]}:=1\in\Gamma(E[D],\Ocal_{E[D]})$ for the section corresponding to the identity element of the ring $\Gamma(E[D],\Ocal_{E[D]})$. Since both sections are horizontal, we will view them as elements of $\HdRabs{0}{E[D]}$. \eqref{eq_Syn_seq} gives us the exact sequence
\[
\begin{tikzcd}[row sep=tiny]
	0\ar[r] & \varprojlim_n \Hsynabs{0}{\Ee[D],\Logsyn{n}}\ar[r] & \varprojlim_n F^0\HdRabs{0}{E[D]_K,\Logn}\ar[r,"1-\phi"] & \phantom{A}\\
	& \phantom{\varprojlim_n \Hsynabs{0}{\Ee[D],\Logsyn{n}}}\ar[r,"1-\phi"] & \varprojlim_n \Hrigabs{0}{E_k[D],(\Logn)^\rig}. &
\end{tikzcd}
\]
This sequence allows us to view $D^21_{e}-1_{E[D]}\in \ker(1-\phi)$ as element of $\varprojlim_n \Hsynabs{0}{\Ee[D],\Logsyn{n}}$.
\begin{defin}
	The $D$-variant of the syntomic polylogarithm is the unique pro-system
	\[
		\polsyn\in\varprojlim_n \Hsynabs{1}{\Uu_D,\Logsyn{n}(1)}
	\]
	mapping to $D^2 1_{e}-1_{E[D]}$ under the residue map in the localization sequence.
\end{defin}

\section{The de Rham realization of the elliptic polylogarithm}
By forgetting the filtration and the Frobenius structure of the syntomic logarithm class $[\Logsyn{1}]$ we obtain an extension class in the category $\VIC{E/K}$ of vector bundles with integrable $K$-connections
\[
	[\LogsyndR{1},\nabla]\in \Ext^1_{\VIC{E/K}}(\Ocal_E,\Hcal_E).
\]
The extension class $[\LogsyndR{1},\nabla]$ maps to zero under $e^*$ respectively to the identity in the short exact sequence
\[
\begin{tikzcd}[column sep=small]
	0\ar[r] & \Ext^1_{\VIC{S/\KK}}(\Ocal_S,\Hcal)\ar[r,shift left,"\pi^*"] & \ar[l,shift left,"e^*"]\Ext^1_{\VIC{E/\KK}}(\Ocal_E,\Hcal_E) \ar[r] & \Hom_{\VIC{S/\KK}}(\Hcal,\Hcal)\ar[r] & 0.
\end{tikzcd}
\]
Contrary to the syntomic realization, an extension representing the first de Rham logarithm class might have non-trivial automorphisms but the additional datum of the section $\idM^{(1)}$ distinguishes a splitting of $e^*\Log{1}$ and thereby rigidifies the situation. The triple $(\Log{1},\nabla_{\Log{1}},\idM^{(1)})$ is called \emph{the first de Rham logarithm sheaf}. For more details on the de Rham logarithm sheaves we refer the reader to \cite[\S2,\S3]{deRham}. The triple $(\Log{n},\nabla,\idM^{(n)})$ obtained by taking $n$-th tensor symmetric powers will be called \emph{$n$-th de Rham logarithm sheaves}. The elliptic polylogarithm in algebraic de Rham cohomology
\[
	\polD\in\lim_n \HdR{1}{U_D/K,\Log{n}}
\]
 is then defined in complete analogy using the localization sequence in de Rham cohomology, i.e. $\polD$ is the unique pro-system mapping to the horizontal section $D^2 1_{e}-1_{E[D]}$ under the residue map in the exact sequence
\[
		\begin{tikzcd}[column sep=small]
			0\ar[r] & \varprojlim_n \HdRabs{1}{U_D/\KK, \Logn}\ar[r,hook,"\Res"] & \prod_{k=0}^\infty \HdRabs{0}{E[D]/\KK,\Sym^k\Hcal_{E[D]}}\ar[r,"\sigma"] & K\ar[r] & 0.
		\end{tikzcd}
	\] 
According to a theorem of R. Scheider, the de Rham logarithm sheaves of an elliptic curve are given by restricting the Poincar\'e bundle on the universal vectorial extension to infinitesimal neighbourhoods of the elliptic curve. Let us briefly recall his construction here and let us refer to \cite{rene} or \cite[\S4]{deRham} for details. Let $\Po$ be the Poincar\'e bundle on $E\times_S E^\vee$ and let $E^\dagger\rightarrow E^\vee$ be the universal vectorial extension of $E^\vee$. There is a universal integrable $E^\dagger$-connection $\con{\dagger}$ on the pullback $\Po^\dagger$ of $\Po$ along $E\times_S E^\dagger\rightarrow E\times_S\Ed$. The connection $\con{\dagger}$ induces an integrable $S$-connection $\con{\Ln^\dagger}$ on
\[
	\Ln^\dagger:=(\pr_E)_* \left( \Po^\dagger\Big|_{E\times_S\Inf^n_e E^\dagger}\right).
\]
Further, the canonical rigidification of the Poincar\'e bundle induces an isomorphism of $\Ocal_S$-modules
\[
	e^* \Ln^\dagger \righteq \Ocal_{\Inf^n_e E^\dagger}.
\]
In particular, there is a distinguished section $1\in \Gamma(S,e^*\Ln^\dagger)$. The following Theorem is due to Scheider. For a much shorter proof we refer to \cite[Corollary 4.6]{deRham}.
\begin{thm}[Scheider, \cite{rene}]
	There is a unique prolongation $\nabla_{\Ln^\dagger}^{abs}$ of the $S$-connection $\con{\Ln^\dagger}$ to a $K$-connection such that
	\[
		(\Ln^\dagger,\nabla_{\Ln^\dagger}^{abs}, 1)
	\]
	is an explicit model for the $n$-th de Rham logarithm sheaf.
\end{thm}

By its very definition $\Log{1}$ is an extension of $\Ocal_E$ by $\Hcal_E$. Viewing $\Ocal_E$ as sitting in filtration step $0$, there is a unique Hodge filtration $F^\bullet$ of $\Log{1}$ compatible with the Hodge filtrations of $\Hcal_E$ and $\Ocal_E$. The filtration on $\Log{n}$ is defined by passing to symmetric powers. Let us write $\Log{n}(1)$ for $\Log{n}$ with shifted filtration $F^{\bullet+1}$. Building on the result of Scheider, we have constructed in \cite[\S5]{deRham} $1$-forms with values in the logarithm sheaves
\[
	\LnD\in\Gamma(E,\Omega^1_{E/S}(E[D])\otimes_{\Ocal_E} F^0 \Ln^\dagger(1)). 
\] 
Let us briefly indicate the construction of $\LnD$ and let us refer to \cite[\S5]{deRham} for details: The starting point of our construction is the \emph{Kronecker section}, which is defined as follows: The autoduality isomorphism
\begin{equation}\label{EP_eqAutodual}
\begin{tikzcd}[row sep=tiny]
	\lambda: E \ar[r] & \Pic=:\Ed\\
	P\ar[r,mapsto] & {[\Ocal_E([-P]-[e])]}
\end{tikzcd}
\end{equation}
gives an explicit model for the rigidified Poincar\'e bundle
\begin{align*}
	(\Po,r_0):&=\left( \pr_1^* \Ocal_E([e])^{\otimes-1}\otimes\pr_2^*\Ocal_{E}([e])^{\otimes-1}\otimes \mu^*\Ocal_E([e])\otimes \pi_{E\times E}^* \om_{E/S}^{\otimes -1} ,r_0\right).
\end{align*}
Here, $\Delta=\ker \left(\mu:E\times \Ed\rightarrow E\right)$ is the anti-diagonal and $r_0$ is the rigidification $(e\times \id)^*\Po\cong \Ocal_{\Ed}$ induced by the canonical isomorphism 
\[
	e^*\Ocal_E(-[e])\righteq \om_{E/S}:=e^*\Omega^1_{E/S}.
\]
This description of the Poincar\'{e} bundle gives the following isomorphism of locally free $\Ocal_{E\times E}$-modules, i.\,e. all tensor products are taken over $\Ocal_{E\times E}$:
\begin{align}\label{ch_EP_eq5}
	\Po\otimes \Po^{\otimes -1} \cong \Po \otimes \Omega^1_{E\times E/E}([e\times E]+[E\times e]) \otimes \Ocal_{E\times E}( - \Delta)
\end{align}
The line bundle $\Ocal_{E\times E}( - \Delta)$ can be identified with the ideal sheaf $\Jcal_\Delta$ of the anti-diagonal $\Delta$ in $E\times_S E$ in a canonical way. If we combine the inclusion 
\[
	\Ocal_{E\times E}( - \Delta)\cong \Jcal_\Delta\hookrightarrow \Ocal_{E\times E}
\]
with \eqref{ch_EP_eq5}, we get a morphism of $\Ocal_{E\times E}$-modules
\begin{equation}\label{ch_EP_eq6}
	\Po\otimes \Po^{\otimes -1} \hookrightarrow \Po \otimes \Omega^1_{E\times E/E}([e\times E]+[E\times e]).
\end{equation}
The \emph{Kronecker section} 
	\[
		s_{\mathrm{can}}\in \Gamma\left(E\times_S E^\vee, \Po \otimes_{\Ocal_{E\times E^\vee}} \Omega^1_{E\times E^\vee/E^\vee}([e\times E^\vee]+[E\times e])\right)
	\]
is then defined as the image of the identity section $\id_{\Po}\in\Gamma(E\times E, \Po\otimes \Po^{\otimes -1})$ under \eqref{ch_EP_eq6}.\par 

The universal property of the Poincar\'e bundle gives us a canonical isomorphisms for $D>1$:
\[
	\gamma_{1,D}\colon (\id\times [D])^*\Pcal\righteq ([D]\times \id)^*\Pcal
\]
Let us define the \emph{$D$-variant of the Kronecker section} by
\[
			\scan^D:= D^2\cdot  \gamma_{1,D}\left( (\id \times [D])^*(\scan)\right)-([D]\times \id)^*(\scan).
\]
The restriction of $\scan^D$ along $E\times_S\Inf^n_e E^\vee\hookrightarrow E\times_S E^\vee$ induces a section
\[
	\lnD\in \Gamma\left(E,\Ln^\dagger\otimes_E \Omega^1_{E/S}(E[D])\right).
\]
This almost yields the desired $1$-forms $\LnD$ except, that $\lnD$ are only $1$-forms relative $S$. But it is possible to lift these $1$-forms in a canonical way to the absolute $1$-forms
\[
	\LnD\in \Gamma\left(E,\Ln^\dagger\otimes_E \Omega^1_{E/K}(E[D])\right).
\]
\begin{thm}[{\cite[Theorem 5.7]{deRham}}]
The system of $1$-forms $(\LnD)_{n\geq 0}$ represents the polylogarithm in algebraic de Rham cohomology
	\[
		([\LnD])_{n\geq 0}=\polD\in \lim_n \HdR{1}{U_D/S,\Log{n}(1)}.
	\]
\end{thm}
Combining this with \cref{prop_BK}, has the following immediate consequence for the syntomic realization:
\begin{cor}\label{cor_synrepresentative}
	Let $\pi:\Ee\rightarrow \Ss$ be a morphism of syntomic data of elliptic curves.	There is a unique system of overconvergent sections $\rho_n\in\Gamma(\bar{\Ee}_K,j_D^\dagger(\Ln^\dagger))$ satisfying the differential equation
	\[
		\nabla_{\Logsyn{n}}(\rho_n)=(1-\Phi)(\LnD).
	\]
	The pair $(\rho_n,\LnD)$ is the unique pair representing the syntomic realization of the polylogarithm:
	\[
		([\rho_n,\LnD])_{n\geq 0}=\polsyn \in \varprojlim_n \Hsynabs{1}{\Uu_D,\Logsyn{n}}.
	\]
\end{cor}

In the following we would like to give an explicit description of the overconvergent sections $\rho_n$ which describe the Frobenius structure on the polylogarithmic extension in syntomic cohomology.

\section{The ordinary locus of the modular curve}
Let $p$ be a prime and $N>3$ be an integer prime to $p$. Let $K=\Qp$ and denote by $\Vv$ the smooth pair $\Vv=(\Spec \Zp,\Spec\Zp)$. For the modular curve $M=M_{N,\Zp}$ with $\Gamma(N)$-level structure over $\Zp$ choose a smooth compactification $\bar{M}$ and let $(E=E_{N},\alpha_N)$ be the universal elliptic curve with level $N$-structure over $M$. Let $\bar{E}$ be the Neron model of $E$ over $\bar{M}$. Then
\[
	(E,\bar{E})\xrightarrow[]{\pi} (M,\bar{M})
\]
is a smooth proper morphism of smooth pairs. If we restrict to the ordinary locus $E^\ord\subseteq E$, we obtain a smooth proper morphism of smooth pairs:
\[
	(E^\ord,\bar{E})\rightarrow (M^\ord,\bar{M})
\]
Let $\Ecal^\ord,\bar{\Ecal}$ resp. $\Mcal^\ord,\bar{\Mcal}$ be the formal completion of $E^\ord,\bar{E}$ resp. $M^\ord,\bar{M}$ with respect to the special fiber. Then, $\Mcal^\ord$ classifies ordinary elliptic curves with level $N$-structure over $p$-adic rings. If we divide an ordinary elliptic curve with level $N$-structure $(E,\alpha)$ by its canonical subgroup, we obtain another ordinary elliptic curve $(E'=E/C,\alpha')$ with level $N$-structure. In particular, the map $(E,\alpha)\mapsto(E/C,\alpha')$ induces a map
\[
	\Frob: \Mcal^\ord\rightarrow\Mcal^\ord
\]
lifting the Frobenius morphism on the special fiber. By \cite[Chapter 3]{katz_padicmf} the induced Frobenius $\Mcal^\ord_{\Qp}\rightarrow \Mcal^\ord_{\Qp}$ on the associated rigid analytic space $\Mcal^\ord_{\Qp}$ is overconvergent. In particular, we have a canonical overconvergent Frobenius $\phi_M$ on the smooth pair $(M^\ord,\bar{M})$. The associated syntomic datum will be denoted by
\[
	\Mm^\ord:=(M^\ord,\bar{M},\phi_M).
\]
For the moment let us write $E^\ord|_{\Mcal^\ord}$ for the pullback of the ordinary elliptic curve to the formal completion. Then, $E^\ord|_{\Mcal^\ord}$ is the universal ordinary elliptic curve with level structure over $p$-adic rings. The commutative diagram
\begin{equation*}
	\begin{tikzcd}
		E^\ord|_{\Mcal^\ord}\ar[rd,"\pi"] \ar[r,"\varphi"] & E'|_{\Mcal^\ord}:=(E^\ord/C)|_{\Mcal^\ord}\ar[d]\ar[dr, phantom, "\lrcorner", very near start] \ar[r,"\widetilde{\Frob}"] & E^\ord|_{\Mcal^\ord}\ar[d,"\pi"]\\
		& \Mcal^\ord\ar[r,"\Frob"] & \Mcal^\ord
	\end{tikzcd}
\end{equation*}
induces a Frobenius lift $E^\ord|_{\Mcal^\ord}\rightarrow E^\ord|_{\Mcal^\ord}$ on $E^\ord|_{\Mcal^\ord}$ which gives us a canonical overconvergent Frobenius $\phi_E$ on the smooth pair $(E^\ord,\bar{E})$. The associated syntomic datum is
\[
	\Ee:=(E^\ord,\bar{E},\phi_E).
\]
Let us write $\pi:\Ee\rightarrow\Mm$ for the corresponding morphism of syntomic data.

Let us now turn our attention to the moduli space of trivialized elliptic curves as defined by Katz \cite{katz_padicinterpol}. Let $E/S=\Spec R$ be an elliptic curve over a $p$-adic ring $R$. A trivialization of $E$ is an isomorphism
	 \[
	 	\beta: \Ef\righteq \Gmf{S}
	 \]
	 of formal groups over $R$. For $N\geq 1$ a natural number coprime to $p$, a trivialized elliptic curve with $\Gamma(N)$-level structure is a triple $(E,\beta,\alpha_N)$ consisting of an elliptic curve $E/S$ a rigidification $\beta$ and a level structure $\alpha_N:(\ZZ/N\ZZ)^2_S\righteq E[N]$. Let $(\Etriv,\beta,\alpha_N)$ be the universal trivialized elliptic curve with $\Gamma(N)$-level structure over $\Mtriv=\Spec \VpN$. For more details we refer to \cite[Ch. V]{katz_padicinterpol}. The ring $\VpN$ will be called ring of generalized $p$-adic modular forms. Let us write $\Mftriv=\Spf \VpN$ for the formal completion of the moduli space $M^\triv=\Spec \VpN$ along its special fiber. The existence of a trivialization on an elliptic curve already implies that the curve is ordinary. Thus, the forgetful map
\[
	(E,\alpha,\beta)\mapsto (E,\alpha)
\]
induces a map $\Mftriv\rightarrow\Mcal^\ord$. The induced map on rigid analytic spaces sits in the following Cartesian diagram
\[
	\begin{tikzcd}
		\Ecal^\triv_{\Qp} \ar[r,"\tilde{p}"]\ar[d] & \Ecal^\ord_{\Qp}\ar[d]\\
		\Mcal^\triv_{\Qp} \ar[r]  & \Mcal^\ord_{\Qp}.
	\end{tikzcd}
\]
For $(a,b)\in(\ZZ/N\ZZ)^2$ let $t=t_{a,b}$ respectively $\tilde{t}=\tilde{t}_{a,b}$ be the associated torsion sections on $E^\ord$ respectively $\Etriv$. Let us furthermore write $]\tilde{t}[$ for the tubular neighbourhood of the reduction of $\tilde{t}$ in $\Ecal^\triv_{\Qp}$. Pullback along the covering map $\tilde{p}$ induces an injection
\[
	\Gamma(]t[, (\Lcal_{n,E^\ord}^\dagger)^\rig) \hookrightarrow \Gamma(]\tilde{t}[, (\Lcal_{n,\Etriv}^\dagger)^\rig), \quad \sigma\mapsto \tilde{p}^*\sigma.
\]
The main goal of the rest of the paper is to give an explicit description of the overconvergent sections $\rho_n$ appearing in the description of the syntomic polylogarithm in \cref{cor_synrepresentative}. The advantage of describing $\tilde{p}^*(\rho_n)$ instead of $\rho_n$ is that we will construct a canonical basis $\hat{\omega}^{[k,l]}$ of $(\Lcal_{n,\Etriv}^\dagger)^\rig$ in \cref{sec_splitting}. This basis allows us to describe sections of $\Lcal_{n}^\dagger$ on tubular neighbourhoods of torsion sections explicitly.

\section{An explicit model for the syntomic the logarithm sheaves for ordinary elliptic curves}
As before let $\Ee^\ord\rightarrow \Mm^\ord$ be the syntomic datum associated to the ordinary part of the modular curve of level $\Gamma(N)$. First let us give a complete description of the syntomic logarithm sheaves in terms of the Poincar\'e bundle. Recall, that the tuple
\[
	(\Ln^\dagger,\nabla_{\Ln^\dagger}^{abs},F^\bullet, 1)
\]
provides an explicit model for the de Rham part of the syntomic logarithm sheaves in terms of the Poincar\'e bundle. It remains to give an explicit description of the Frobenius structure.Let us write $\varphi\colon E^\ord\twoheadrightarrow E'=E^\ord/C$ for the isogeny given by the quotient of $E^\ord$ by its canonical subgroup. Recall, that the universal vectorial extension $E^{\ord,\dagger}$ of $E^{\ord,\vee}$ classifies line bundles of degree zero with an integrable connection. In particular, there is a unique map $\varphi^\dagger: E'^\dagger\rightarrow E^{\ord,\dagger} $ together with a unique horizontal isomorphism
\begin{equation}\label{eq_Poincare_isogeny}
	(\id_{E^\ord}\times \varphi^\dagger )^*\Po^\dagger \righteq (\varphi\times \id_{E'^\dagger})^*\Po'^\dagger.
\end{equation}
classifying the line bundle $(\varphi\times \id_{E'^\dagger})^*\Po'^\dagger$ with its pullback connection. The restriction of $\varphi^\dagger: E'^\dagger\rightarrow E^{\ord,\dagger}$ to the $n$-th infinitesimal neighbourhood of the identity induces an isomorphism:
\[
	\varphi^\dagger|_{\Inf^n_e E_K'^\dagger}\colon \Inf^n_e E'^\dagger_K\righteq \Inf^n_e{E^{\ord,\dagger}_K}.
\]
Restricting the map \eqref{eq_Poincare_isogeny} along $E_K\times \Inf^n_e E^\dagger_K$ gives an horizontal isomorphism
\begin{equation}\label{eq_Ln_phi}
	\Lcal_{n,E_K}^\dagger \righteq \varphi^*\Lcal_{n,E_K'}^\dagger.
\end{equation}
The reader familiar with the properties of the logarithm sheaves might have noticed, that this isomorphism is nothing than the observation that the logarithm sheaves are invariant under isogenies. Now, let us once again consider the diagram
\begin{equation*}
	\begin{tikzcd}
		E^\ord|_{\Mcal^\ord}\ar[rd,"\pi"] \ar[r,"\varphi"] & E'|_{\Mcal^\ord}:=(E^\ord/C)|_{\Mcal^\ord}\ar[d]\ar[dr, phantom, "\lrcorner", very near start] \ar[r,"\widetilde{\Frob}"] & E^\ord|_{\Mcal^\ord}\ar[d,"\pi"]\\
		& \Mcal^\ord\ar[r,"\Frob"] & \Mcal^\ord
	\end{tikzcd}
\end{equation*}
Recall, that the composition $\phi_E:=\widetilde{\Frob}\circ \varphi$ provides a canonical Frobenius lift on $E^\ord$. 
Since the Poincar\'e bundle is compatible with base change we have an horizontal isomorphism
\begin{equation}\label{eq_Ln_phi2}
	\widetilde{\Frob}^* \Ln^\dagger \righteq \Lcal_{n,E'}^\dagger.
\end{equation}
Thus, combining \eqref{eq_Ln_phi} and \eqref{eq_Ln_phi2} gives an horizontal isomorphism
\[
	\Lcal^\dagger_{n,E_K^\ord} \righteq \phi_E^* \Lcal^\dagger_{n,E_K^\ord}.
\]
The inverse of this map induces the desired Frobenius structure
\[
	\Phi_{\Ln^\dagger}\colon \phi_E^* (\Ln^\dagger)^\rig \righteq  (\Ln^\dagger)^\rig.
\]
This completes our full description of the syntomic logarithm sheaves in terms of the Poincar\'e bundle:
\begin{prop}
	The filtered overconvergent $F$-isocrystal
	\[
		(\Ln^\dagger,\nabla_{\Ln^\dagger}^{abs},F^\bullet,\Phi_{\Ln^\dagger},1)
	\]
	is an explicit model for the syntomic logarithm sheaf.
\end{prop}
\begin{proof}
	It only remains to prove that the constructed Frobenius structure coincides with the abstractly defined Frobenius structure of the syntomic logarithm sheaves. But it follows immediately from the definition of $\Phi_{\Ln^\dagger}$ that $e^*\Phi_{\Ln^\dagger}(1)=1$ where $1\in\Gamma(S,e^*\Ln^\dagger)$ is the fixed section in the above datum. Since there is only one horizontal morphism with this property, we see that $\Phi_{\Ln^\dagger}$ coincides with the abstractly defined Frobenius structure on the syntomic logarithm sheaf.
\end{proof}

\section{$p$-adic theta functions and the $p$-adic Eisenstein measure}
In this section we recall our approach towards $p$-adic interpolation of Eisenstein--Kronecker series via $p$-adic theta functions, see \cite[Part II]{EisensteinPoincare}.
\subsection{$p$-adic Eisenstein--Kronecker series}
Let us first recall Norman's definition of the $p$-adic Eisenstein--Kronecker series. Classical modular forms of weight $k$ and level $\Gamma(N)$, can be seen more geometrically as sections of the $k$-th tensor power of the cotangent-sheaf $\omega_{E(\CC)/M(\CC)}^{\otimes k}$ of the complex universal elliptic curve of level $\Gamma(N)$. This leads in a natural way to the definition of geometric modular forms and allows to study modular forms from an algebraic perspective. More generally, a certain class of $\Ccal^\infty$-modular forms, the \emph{quasi-holomorphic modular forms} allows a similar interpretation. For more on \emph{quasi-holomorphic modular forms} and there geometric interpretation let us refer to \cite[\S 2]{urban}. It is possible to see quasi-holomorphic modular forms (of level $\Gamma(N)$, weight $k$ and order $r$) as sections
\[
	F^{k-r} \Sym^k \Hcal^1_{\mathrm{dR}}
\]
sitting in a certain filtration step of the Hodge filtration of symmetric powers of the relative de Rham cohomology $\Hcal^1_{\mathrm{dR}}:=R^1\pi_* \Omega^\bullet_{E/M}$. The link back to the associated $\Ccal^\infty$-modular forms comes from the Hodge decomposition
\[
	\Hcal^1_{\mathrm{dR}}(\Ccal^\infty)\righteq \omega(\Ccal^\infty)\oplus\bar{\omega}(\Ccal^\infty)
\]
which is non-holomorphic. \emph{Eisenstein--Kronecker series} provide a particular class of nearly-holomorphic modular forms of number theoretic interest. For a given lattice $\Gamma=\ZZ+\tau\ZZ\subseteq \CC$ and $t\in \frac{1}{N}\Gamma,t'\in\frac{1}{D}\Gamma$ let us consider the series
\[
	\tilde{e}_{a,b}(t,t';\tau):=\frac{(-1)^{a+b-1}(b-1)!}{A(\Gamma)^a} \sum_{\gamma\in\Gamma\setminus\{-t\}} \frac{(\bar{t}+\bar{\gamma})^a}{(t+\gamma)^b} \langle \gamma,t'  \rangle
\]
with $\langle z,w  \rangle:=\exp\left( \frac{z\bar{w}-w\bar{z}}{A(\Gamma)} \right)$ and $A(\Gamma)=\frac{\tau-\bar{\tau}}{2 \pi i} $. This series converges absolutely for $b>a+2$; for arbitrary integers $a,b$ it can be defined by analytic continuation, c.f.~\cite{bannai_kobayashi}. The above mentioned geometric interpretation of nearly-holomorphic modular forms is very useful for studying algebraic and $p$-adic properties of Eisenstein-Kronecker series. Indeed, it is possible to associate in a functorial way to every test object $(E/S,t,t')$ with $E/S$ an elliptic curve and $t\in E[N](S),t'\in E[D](S)$ torsion sections certain elements
\[
	E^{a,b}_{t,t'}\in \Gamma(S,\Sym^{k+r}\HdR{1}{E/S})
\]
which correspond to the classical analytic Eisenstein--Kronecker series $\tilde{e}_{a,b}(t,t')$ via the Hodge decomposition on the universal elliptic curve. This purely algebraic interpretation of real analytic Eisenstein series goes back to a construction of Katz involving the Gauss-Manin connection on the modular curve. An alternative construction can be given using the Poincar\'e bundle on the universal vectorial extension of an elliptic curve, cf.~\cite[\S 4]{EisensteinPoincare}. For the construction of $E^{a,b}_{t,t'}$ and for the discussion of their properties we refer to \cite[\S 4]{EisensteinPoincare}. While studying the syntomic realization of the elliptic polylogarithm, the following variant of the geometric Eisenstein--Kronecker series will appear naturally: For a test object $(E/S,t\in E[N](S))$ and a fixed integer $D>1$ let us consider
\[
	\EisD:=\sum_{0\neq t'\in E[D](S)}E^{k,r+1}_{t,t'}\in \Gamma(S,\Sym^{k+r+1}\HdR{1}{E/S}).
\]
Once the algebraic sections $\EisD$ are defined, it is straightforward to define Eisenstein--Kronecker series $p$-adically. Instead of applying the Hodge decomposition one can use the unit root decomposition on the universal trivialized elliptic curve to construct generalized $p$-adic modular forms: Let $\Etriv/\Mtriv$ be the universal trivialized elliptic curve and for $(a,b)\in (\ZZ/N\ZZ)^2$ let $t=t_{a,b}\in \Etriv[N](\Mtriv)$ be the associated $N$-torsion section. Applying the unit root decomposition
\[
	\Sym^{k+r+1}\HdR{1}{\Etriv/\Mtriv}\twoheadrightarrow \Sym^{k+r+1} \om_{\Etriv/\Mtriv}\righteq \Ocal_{\Mtriv} 
\]
followed by the trivialization $\om_{\Etriv/\Mtriv}=\beta^*\left(\frac{\dd T}{1+T}\right)\cdot \Ocal_{\Mtriv}$ to the sections $\EisD$ yields generalized $p$-adic modular forms
\[
	\EispD\in \VpN=\Gamma(\Mtriv,\Ocal_{\Mtriv}).
\]
Let us call these $p$-adic modular forms \emph{$p$-adic Eisenstein--Kronecker series}.

\subsection{$p$-adic theta functions of the Poincar\'e bundle}\label{subsec_padictheta}
Let us briefly recall the construction of $p$-adic theta functions for sections of the Poincar\'e bundle. For details we refer to \cite[\S 6]{EisensteinPoincare}. Let $E/S$ be an elliptic curve over a $p$-adic ring $S=\Spec R$ with fiberwise ordinary reduction. Let us write
\[
	i_n\colon C_n:=E[p^n]_0\hookrightarrow E,\quad j_n\colon D_n:=E^\vee[p^n]_0\hookrightarrow E^\vee
\]
for the inclusion of the connected components of the $p^n$-torsion groups. We define
\[
	\varphi_n\colon E\twoheadrightarrow E/C_n=:E'
\]
and write $D_n':=(E')^\vee[p^n]_0$ for the connected of the $p^n$-torsion of $(E')^\vee$. Since $\varphi_n^\vee\colon (E')^\vee\rightarrow E^\vee$ is \'etale, it induces an isomorphism $D'_n\righteq D_n$. Let us write $\Po$ for the Poincar\'e bundle on $E\times_S\Ed$ and $\Po'$ for the Poincar\'e bundle on $E'\times_S (E')^\vee$. By pullback of the Poincar\'e bundle $\Po$ along the commutative diagram
\[
	\begin{tikzcd}
		D'_n\ar[r,"j'_n",hook]\ar[d,"\cong"] & (E')^\vee\ar[d,"\varphi_n^\vee"]\\
		D_n\ar[r,"j_n",hook] & E^\vee
	\end{tikzcd}
\]
and using the isomorphism $(\id\times \varphi^\vee_n)^*\Po\righteq (\varphi_n\times\id)^*\Po'$ we obtain an isomorphism
\[
	(i_n\times j_n)^*\Po \righteq (i_n\times j'_n)^*(\id\times\varphi_n^\vee)^*\Po\righteq (i_n\times j'_n)^*(\varphi_n\times \id)^*\Po'.
\]
On the other hand, $\varphi_n\circ i_n$ factors through the zero section and we can use the canonical rigidification of the Poincar\'e bundle $(e\times\id)\Po'\righteq \Ocal_{(E')^\vee}$ to deduce the isomorphism of $\Ocal_S$-modules
\[
	(i_n\times j_n)^*\Po \righteq (\varphi_n\circ i_n\times j'_n)^*\Po'\righteq \Ocal_{C_n}\otimes_{\Ocal_S} \Ocal_{D_n}.
\]
By passing to the limit, we obtain the desired isomorphism
\[
	\Po|_{\Ef\times_S\Ef^\vee}\righteq \Ocal_{\Ef\times_S\Ef^\vee}.
\]
Slightly more generally, we can define a trivialization at every $N$-torsion section $t\in E[N](S)$ for $N$ coprime to $p$ as follows: The canonical isomorphism
\[
	([N]\times\id)^*\Po\righteq (\id\times[N])^*\Po 
\]
induces an isomorphism
\begin{equation}\label{eq_translation}
	(T_t\times [N])^*\Po \righteq (T_t\times \id)^*([N]\times \id)^*\Po=([N]\times \id)^*\Po\righteq (\id\times [N]^*\Po).
\end{equation}
Since $N$ is coprime to $p$, the map $[N]\colon \Ef^\vee\rightarrow \Ef^\vee$ is an isomorphism. Restricting \eqref{eq_translation} along $\Ef\times_S\Ef^\vee$ allows us to define a trivialization of the Poincar\'e bundle infinitesimally around torsion sections:
\begin{equation}\label{eq_Poincare_translation}
	\Po|_{\Ef_t\times_S \Ef^\vee}\righteq \Po|_{\Ef\times_S\Ef^\vee}\righteq \Ocal_{\Ef\times_S\Ef^\vee}.
\end{equation}
This trivialization allows us to define $p$-adic theta functions for sections of the Poincar\'e bundle. Let us apply this to the universal trivialized elliptic curve $\Etriv/\Mtriv$ of level $\Gamma(N)$ with $N$ co-prime to $p$. Let us recall that $\VpN:=\Gamma(\Mtriv,\Ocal_{\Mtriv})$ is the ring of \emph{generalized $p$-adic modular forms}. For $(a,b)\in (\ZZ/N\ZZ)^2$ let us write $\tilde{t}:=\tilde{t}_{(a,b)}$ for the associated $N$-torsion section. The trivialization gives us an isomorphism
\[
	\Gamma(\Eftriv_t\times_{\Mtriv}\Edftriv,\Po|_{\Eftriv\times\Edftriv})\righteq \VpN\llbracket S,T\rrbracket.
\]
If $U\subseteq \Etriv\times_{\Mtriv}\Eftriv$ is an open neighbourhood of the torsion section $\tilde{t}$ and $\sigma\in\Gamma(U,\Po)$, let us write
\[
	\theta_{\tilde{t}}(\sigma)\in \VpN\llbracket S,T \rrbracket
\]
for the image of $\sigma|_{\Eftriv_{\tilde{t}}\times\Edftriv}$ under the above isomorphism. We call $\theta_{(a,b)}(\sigma)$ the \emph{$p$-adic theta function associated to the section $\sigma$ at $\tilde{t}$}. Of particular interest for us is the $p$-adic theta function associated to the Kronecker section $\scan^D$. Let us write
\begin{equation}\label{eq_ptheta}
	\pthetaD_{\tilde{t}}:=\pthetaD_{(a,b)}:=\theta_{\tilde{t}}(\scan^D)
\end{equation}
for the $p$-adic theta function associated to the Kronecker section at $\tilde{t}=\tilde{t}_{(a,b)}$. In the next section, we will see that the $p$-adic theta function $\pthetaD_{(a,b)}$ is closely related to $p$-adic Eisenstein--Kronecker series.

\subsection{$p$-adic theta functions and $p$-adic Eisenstein--Kronecker series}
 Let us now turn our attention to the $p$-adic interpolation of the $p$-adic Eisenstein--Kronecker series. Let us refer to \cite[\S8]{EisensteinPoincare} for details: For a $p$-adic ring $R$ the \emph{Amice transform} is an isomorphism of $R$-algebras
\[
	R\llbracket S_1,...,S_n \rrbracket\righteq \Meas(\Zp^n,R),\quad f\mapsto \mu_f
\]
between the ring of power series over $R $ and the ring of $p$-adic measures with values in $R$. It is uniquely characterized by the property
\[
	\int_{\Zp^n}x_1^{k_1}\cdot ... x_n^{k_n}d \mu_f(x_1,...,x_n)=\partial_1^{\circ k_1}...\partial_n^{\circ k_n} f|_{S_1=...=S_n=0}
\]
where $\partial_i:=(1+S_i)\frac{\partial}{\partial S_i}$ is the invariant derivation associated to the coordinate $S_i$ on $\Gmf{R}^n$. It turns out that the Amice transform of the $p$-adic theta function $\pthetaD_{(a,b)}\in \VpN\llbracket S,T\rrbracket$ provides a measure $\muEisD$ which interpolates the Eisenstein--Kronecker series $p$-adically:
\begin{thm}[{\cite[Cor. 8.2]{EisensteinPoincare}}]
\[
	\EispD=\int_{\Zp^2} x^ky^r d\muEisD.
\]
\end{thm}

Let us view the restricted measure $\muEisD\Big|_{\Zp^\times\times\Zp}$ again as a measure on $\Zp\times\Zp$ and define $\pthetaDp_{(a,b)}$ as the power series corresponding to it under the Amice transform. As above, let us write $\Frob\colon \Mtriv\rightarrow \Mtriv$ for the Frobenius lift on the moduli space of trivialized elliptic curves induced by taking the quotient by the canonical subgroup $\Etriv\twoheadrightarrow \Etriv/C$. By abuse of notation, let us also write $\Frob^*$ for the map
\[
	\Frob^* \colon\VpN\llbracket S,T\rrbracket\rightarrow \VpN\llbracket S,T\rrbracket
\]
obtained by base change, i.e. $\Frob^*$ acts coefficient-wise on a power-series. We can give the following explicit description of $\pthetaDp_{(a,b)}(S,T)$:
\begin{prop}\label{prop_theta_restricted}
The Amice transform of the restricted measure $\muEisD\Big|_{\Zp^\times\times\Zp}$ is given by the formula
\[
	\pthetaDp_{(a,b)}(S,T)=\pthetaD_{(a,b)}(S,T)-\Frob^* \pthetaD_{(a,b)}([p](S),T).
\]
where $[p](S)$ denotes the $[p]$-series of the formal multiplicative group.
\end{prop}
\begin{proof}
In \cite[Thm 11.1]{EisensteinPoincare} we have proven the formula
\[
		\int_{\Zp^\times\times \Zp} f(x,y)\dd \muEisD=\int_{\Zp\times \Zp} f(x,y)\dd \muEisD-\Frob^* \int_{\Zp\times \Zp} f(p\cdot x,y)\dd \muEisD.
	\]
	By passing to the Amice transform, we get the desired equality.
\end{proof}

\section{The infinitesimal splitting}\label{sec_splitting}
Recall that the logarithm sheaves can be obtained by taking successive extensions of symmetric powers of $\Hcal_E$. These extensions split after pullback along $e$, i.e.
\[
e^*\Ln^\dagger\righteq \bigoplus_{k=0}^n \Sym^k \Hcal.
\]
In this section we will extend this splitting to some infinitesimal neighbourhood of torsion sections. On the universal trivialized elliptic curve we have further a canonical basis of $\Hcal$. This together with the splitting will provide us with a canonical basis of $(\Ln^\dagger)^{rig}$ in tubular neighbourhoods of torsion sections.\par 
\subsection{Basic properties}
Let $E/S$ be an elliptic curve over a $p$-adic ring $S=\Spec R$ with fiber-wise ordinary reduction. Let us first recall that $\Ln^\dagger$ was defined by restriction of the Poincar\'e bundle $\Po^\dagger$ to an infinitessimal thickening of $E$:
\[
\Ln^\dagger:=(\pr_E)_* \left( \Po^\dagger\Big|_{E\times_S\Inf^n_e E^\dagger}\right).
\]
Similarly, let us define $\Lcal_n$ as the restriction of the classical Poincar\'e bundle on $E\times_S \Ed$ to an infinitessimal thickening of $E$, i.e.
\[
	\Lcal_n:=(\pr_E)_* \left(\Po|_{E\times_S \Inf^n_e \Ed} \right).
\]
Let us write
\[
	\hat{\Lcal}_n:=\Lcal_n|_{\Ef},\quad \hat{\Lcal}^\dagger_n:=\Lcal^\dagger_n|_{\Ef}
\]
for the restriction of $\Ln$ and $\Ln^\dagger$ to $\Ef$. The canonical projection of the universal vectorial extension to the dual elliptic curve $E^\dagger\rightarrow E^\vee$ induces an inclusion $\Ln\hookrightarrow\Ln^\dagger$ which allows us to view $\Ln$ as a sub-bundle of $\Ln^\dagger$. Note, that the connection on $\Ln^\dagger$ does not restrict to a connection on $\Ln$. The $\mathbb{G}_m$-biextension structure of the Poincar\'e bundle allows us to construct co-multiplication maps
\[
	\hat{\Lcal}_n\rightarrow \TSym^n \hat{\Lcal}_1,\quad  \hat{\Lcal}^\dagger_n\rightarrow \TSym^n \hat{\Lcal}^\dagger_1
\]
let us refer to \cite[\S 9.2]{EisensteinPoincare} for their definition. We have the following:
\begin{lem}[{\cite[Lemma 9.1]{EisensteinPoincare}}]
	The co-multiplication maps 
	\[
	\hat{\Lcal}_n\rightarrow \TSym^n \hat{\Lcal}_1,\quad  \hat{\Lcal}^\dagger_n\rightarrow \TSym^n \hat{\Lcal}^\dagger_1
	\]
	are injective and isomorphisms on the generic fiber
	\[
	\hat{\Lcal}_{n,E_{\Qp}}\righteq  \TSym^n \hat{\Lcal}_{1,E_{\Qp}},\quad  \hat{\Lcal}^\dagger_{n,E_{\Qp}}\righteq \TSym^n \hat{\Lcal}^\dagger_{1,E_{\Qp}}
	\]
\end{lem}
 The construction of \cref{subsec_padictheta} provides an isomorphism
\[
	\Pof:=\Po\Big|_{\Ef\times_S\Ef^\vee}\righteq \Ocal_{\Ef\times_S\Ef^\vee}=\Ocal_{\Ef}\hat{\otimes}_{\Ocal_S} \Ocal_{\Ef^\vee} .
\]
The restriction of this trivialization along $\Ef\times_S \Inf^1_e \Ed$ provides the splitting
\begin{equation}
	\hat{\Lcal}_1=\Ocal_{\Ef}\otimes_{\Ocal_S} \Ocal_{\Inf^1_e \Ed}=\Ocal_{\Ef}\otimes_{\Ocal_S}(\Ocal_S\oplus \om_{\Ed/S}).
\end{equation}
Applying $\TSym^n$ to the isomorphism $\hat{\Lcal}_1\cong\Ocal_{\Ef}\otimes_{\Ocal_S}(\Ocal_S\oplus \om_{\Ed/S})$ gives
\[
	\hat{\Lcal}_{n,E_{\Qp}}\righteq  \bigoplus_{k=0}^n \TSym^k \om_{\Ef_{\Qp}}
\]
where we write $\om_{\Ef_{\Qp}}:=\Ocal_{\Ef_{\Qp}}\otimes_{\Ocal_S}\om_{\Ed/S}$.
\subsection{The infinitesimal splitting on the universal trivialized elliptic curve}
 If we apply this to the universal trivialized elliptic curve $\Etriv/\Mtriv$ of level $\Gamma(N)$, we obtain
\[
	\hat{\Lcal}_{n,\Etriv_{\Qp}}\righteq \bigoplus_{k=0}^n \TSym^k\om_{\Eftriv_{\Qp}}.
\]
The isomorphism $\Edftriv\righteq \Eftriv\righteq \Gmf{\Mtriv}$ provides us with a generator $\omega:=\beta^*\left( \frac{\dd T}{1+T}\right)\in \Gamma(\Eftriv,\om_{\Eftriv})$ corresponding to the invariant differential $d T/(1+T)$ on $\Gmfabs$. The tensor symmetric algebra carries a canonical divided power structure $x\mapsto x^{[k]}:=x\otimes...\otimes x$ on $\TSym^{>0}\om_{\Eftriv}$. This allows us to define
\[
	\hat{\omega}^{[k]}:= \omega^{[k]}.
\]
Similarly, the inclusion $\om_{\Edtriv/\Mtriv}\hookrightarrow \Hcal$ gives us a section $[\omega]\in\Gamma(\Mtriv,\Hcal)$. Let $[u]\in \Gamma(\Mtriv,\Hcal)$ be the unique section in the unit root subspace of $\Hcal$ with $\langle [u],[w] \rangle=1$. The pair $([u],[\omega])$ generates $\Hcal$ as $\Ocal_{\Mtriv}$-module and induces sections
\[
	\hat{\omega}^{[k],[l]}:=[\omega]^{[k]}\cdot [u]^{[l]}\in\Gamma (\Eftriv,\TSym^{k+l}\Hcal_{\Eftriv}).
\]
\begin{lem}[{\cite[Lemma 9.2.]{EisensteinPoincare}}]\label{lem_PhiLn}
		We have canonical $\Ocal_{\Eftriv}$-linear decompositions:
		\[
			\widehat{\Lcal}_{n,\Etriv_{\Qp}}\righteq \bigoplus_{k=0}^n \hat{\omega}^{[k]}\cdot \Ocal_{\Eftriv_{\Qp}},\quad \widehat{\Lcal}^\dagger_{n,\Etriv_{\Qp}}\righteq \bigoplus_{k+l\leq n} \hat{\omega}^{[k,l]}\cdot \Ocal_{\Eftriv_{\Qp}}.
		\]
		These decompositions are compatible with the transition maps 
		\[ 
		\widehat{\Lcal}_{n,\Etriv_{\Qp}}\twoheadrightarrow \widehat{\Lcal}_{n-1,\Etriv_{\Qp}},\quad \widehat{\Lcal}^\dagger_{n,\Etriv_{\Qp}}\twoheadrightarrow \widehat{\Lcal}^\dagger_{n-1,\Etriv_{\Qp}}
		\] 
		and the inclusion $\widehat{\Lcal}_{n,\Etriv_{\Qp}}\hookrightarrow \widehat{\Lcal}^\dagger_{n,\Etriv_{\Qp}}$.
\end{lem}
Let us write $\Ecal^\triv_{\Qp}$ for the rigidification of $\Etriv$ and $]\tilde{t}[\subseteq \Ecal^\triv_{\Qp}$ for the tubular neighbourhood in $\Ecal^\triv_{\Qp}$ of the reduction of the torsion section $\tilde{t}$. Let us write $\Eftriv_{\tilde{t}}$ for the formal completion of $\Eftriv$ along ${\tilde{t}}$. Now,
\[
	\Eftriv_{\tilde{t}}\righteq \Eftriv\righteq\Gmf{\Mtriv}
\]
 yields an isomorphism of rigid analytic spaces:
\[
	]\tilde{t}[\righteq B^{-}(0,1)\times \Mcal^\triv_\Qp
\]
The isomorphism \eqref{eq_Poincare_translation} induces
\[
	\Ln|_{\Eftriv_{\tilde{t}}}\righteq \Ln|_{\Eftriv}=\hat{\Ln}.
\]
Combining this with the decomposition
\[
	\hat{\Lcal}_{n,\Etriv_{\Qp}}\righteq \bigoplus_{k=0}^n \TSym^k\om_{\Eftriv_{\Qp}}
\]
and passing to the associated rigid analytic space, we obtain the following:
\begin{cor}\label{prop_infsplit}
	There is a natural isomorphism of $\Ocal^{an}_{B^{-}(0,1)\times \Mcal^\triv_\Qp}$-modules
	\[
		\left(\Lcal_{n,\Etriv_K}^\dagger\right)^{rig}\Big|_{]\tilde{t}[} \righteq \bigoplus_{k+l\leq n} \hat{\omega}^{[k,l]} \Ocal^\an_{B^-(0,1)\times \Mcal^\triv_{\Qp}}.
	\]
	The basis $(\hat{\omega}^{[k,l]})_{0\leq k+l\leq n}$ is compatible with the transition maps of the logarithm sheaves. 
\end{cor}

We can now express the connection on $\hat{\Ln}^\dagger$ explicitly:

\begin{lem}
The connection $\con{\Ln^\dagger}$ on $\hat{\Ln}^\dagger$ is explicitly given under the decomposition
		\[
			 \widehat{\Lcal}^\dagger_{n,\Etriv_{\Qp}}\righteq \bigoplus_{k+l\leq n} \hat{\omega}^{[k,l]}\cdot \Ocal_{\Eftriv_{\Qp}}.
		\]
by the formula
	\[
		\con{\Ln^\dagger}(\hat{\omega}^{[k,l]})=(l+1)\hat{\omega}^{[k,l+1]}\otimes\omega.
	\]
\end{lem}

\begin{proof}
	We have proven in \cite[Lemma 9.4]{EisensteinPoincare} that 
	\[
		\con{\Ln^\dagger}(\hat{\omega}^{[k,l]})=c\cdot (l+1)\hat{\omega}^{[k,l+1]}\otimes \omega
	\]
	for some $c\in\Zp$. With slightly more effort, we can refine this and prove $c=1$. By definition, the basis $\hat{\omega}^{[k,l]}$ is compatible with the co-multiplication maps
	\[
		\widehat{\Lcal}_{n,E_{\Qp}}^\dagger\righteq \TSym^n \widehat{\Lcal}_{1,E_{\Qp}}^\dagger.
	\]	
	Further these co-multiplication maps are horizontal, thus it is enough to prove the case $n=1$. For simplicity let us write $M:=\Mtriv$ and $E=\Etriv$. As in \cite[\S 5.1]{katz_jacobi} let us write $\HdR{\bullet}{\Ef/M}$ for the relative cohomology of the separated completion $\hat{\Omega}^1_{\Ef/M}$ of the classical de Rham complex. Let us write $\VIC{\Ef/M}$ for the category of locally free $\Ocal_{\Ef}$-modules of finite rank equipped with an (integrable) $M$-connection, i.e. modules $\Fcal$ with an $\Ocal_M$-linear derivative
	\[
		\Fcal\rightarrow \hat{\Omega}^1_{\Ef/M}\otimes_{\Ocal_{\Ef}}\Fcal.
	\]
	Similarly, let us write $\VIC{E/M}$ for the category of locally free $\Ocal_{E}$-modules of finite rank equipped with an (integrable) $M$-connection. 	Let us consider $\Ocal_E$ and $\Hcal_E$ as objects of $\VIC{E/M}$ equipped with the trivial connection. To an extension 
	\[
		0\rightarrow \Hcal_E\rightarrow \Fcal\rightarrow \Ocal_E \rightarrow 0
	\]
	we can associate a section $\Gamma(M,\HdR{1}{E/S,\Hcal_E})$ as the image of $1\in\Gamma(S,\Ocal_S)=\Gamma(S,\HdR{0}{E/S})$ under the connecting homomorphism
\[
		0\rightarrow \HdR{0}{E/S,\Hcal_E}\rightarrow \HdR{0}{E/S,\Fcal}\rightarrow \HdR{0}{E/S} \rightarrow \HdR{1}{E/S,\Hcal_E}.
	\]
	This gives us a map
	\[
		\delta\colon \Ext^1_{\VIC{E/M}}\rightarrow \Gamma(M,\HdR{1}{E/S,\Hcal_E})=\Gamma(M,\Hcal\otimes\HdR{1}{E/M}).
	\]
	By the defining property of the logarithm sheaves, we know that $\delta([\Lcal_{1}^\dagger])$ corresponds to the identity section of $\Hcal^\vee\otimes \Hcal$ using the canonical isomorphism $\HdR{1}{E/M}\cong \Hcal^\vee$. Similarly, we have a connecting homomorphism
	\[
		\hat{\delta}\colon \Ext^1_{\VIC{\Ef/M}}\rightarrow \Gamma(M,\HdR{1}{\Ef/S,\Hcal_E})=\Gamma(M,\Hcal\otimes\HdR{1}{\Ef/M}).
	\]
	By restriction along $\Ef\hookrightarrow E$ we get a commutative diagram
	\[
		\begin{tikzcd}
			\Ext^1_{\VIC{E/M}}(\Ocal_E,\Hcal_E)\ar[r,"\delta"]\ar[d] & \Gamma(M,\Hcal\otimes\HdR{1}{E/M})\ar[d]\\
			\Ext^1_{\VIC{\Ef/M}}(\Ocal_{\Ef},\Hcal_{\Ef})\ar[r,"\hat{\delta}"] & \Gamma(M,\Hcal\otimes\HdR{1}{\Ef/M}).
		\end{tikzcd}
	\]
	It is straight-forward to check, that $\hat{\delta}([\widehat{\Lcal}_{1}^\dagger])$ is contained in $\Gamma(M,\Hcal\otimes D(\Ef/M))$ where $D(\Ef/M)\subseteq \HdR{1}{\Ef/S}$ is the submodule of primitive elements \cite[\S 5.1]{katz_jacobi}. There is a natural inclusion $\om_{E/M}\hookrightarrow D(\Ef/M)$ which is in our case an isomorphism since the co-kernel is the co-Lie algebra of the dual $p$-divisible group of $\Ef$ over $R$ \cite[(5.3.2)]{katz_jacobi}. In particular, we can identify the inclusion $D(\Ef/M)\subseteq \HdR{1}{\Ef/M}$ with the natural inclusion $\om_{E/M}\subseteq \HdR{1}{\Ef/M}$. Using this identification, the above commutative diagram and $\delta([\Lcal_{1}^\dagger])=\id_{\Hcal}$ implies
	\[
		\hat{\delta}([\widehat{\Lcal}^\dagger_{1}])=u\otimes [\omega]\in \Gamma(M,\Hcal\otimes\HdR{1}{\Ef/M}).
	\]
	Indeed, under the canonical projection $\Hcal\twoheadrightarrow \om^\vee_{E/M}$ the unit root space is mapped isomorphically onto $\om^\vee_{E/M}$ and the generator $u$ maps onto the dual $\omega^\vee$ of $\omega$. On the other hand, since the basis $\hat{\omega}^{[0,0]},\hat{\omega}^{[0,1]},\hat{\omega}^{[1,0]}$ splits the extension
	\[
		0\rightarrow \Hcal_{\Ef} \rightarrow \widehat{\Lcal}^\dagger_1 \rightarrow \Ocal_{\Ef}\rightarrow 0
	\]
	we can compute the boundary map explicitly by the formula $[\con{\Lcal_1^\dagger}(\hat{\omega}^{[0,0]})]\in \Gamma(M,\HdR{1}{\Ef/M}\otimes\Hcal)$ where
	\[
		[\cdot]\colon \hat{\Omega}^1_{\Ef/M} \rightarrow \HdR{1}{\Ef/M}
	\]
	maps a form to its associated cohomology class. Since we already know $\con{\Lcal_1^\dagger}(\hat{\omega}^{[0,0]})=c\cdot \hat{\omega}^{[0,1]}\otimes \omega$ and since $\hat{\omega}^{[0,1]}:=u$, we get
	\[
		\hat{\delta}([\widehat{\Lcal}^\dagger_{1}])=[\con{\Lcal^\dagger_1}(\hat{\omega}^{[0,0]})]=c\cdot u\otimes [\omega].
	\]
	We conclude $c=1$.\par 

\end{proof}
The Frobenius structure
\[
	\Phi_{\Ln^\dagger}\colon \phi_E^* (\Ln^\dagger)^\rig \righteq (\Ln^\dagger)^\rig
\]
induces a map on global sections
\[
	\Phi\colon \Gamma(\bar{\Ee}_K,j_D^\dagger(\Ln^\dagger))\rightarrow \Gamma(\bar{\Ee}_K,j_D^\dagger(\Ln^\dagger)),\quad \alpha\mapsto \Phi_{\Ln^\dagger}(\phi_E^*\alpha).
\]
Here, note that $j_D^\dagger$ denotes the overconvergent sections functor $j_D^\dagger:=\lim j_{V*}j_V^{-1}$ where $V$ runs over all strict neighbourhoods of $\Ucal_K$ in $\bar{\Ecal}_K$. The Frobenius structure can be expressed in terms of the basis as follows:
\begin{lem}\label{lem_Frobenius_structure}
The Frobenius structure is given by the following formula:
	\[
		\Phi(\hat{\omega}^{[k],[l]})=p^{-l}\hat{\omega}^{[k],[l]}.
	\]
In particular, $\Phi$ acts trivially on $\hat{\Lcal}_n$.
\end{lem}
\begin{proof}
From \cite[Lemma 9.3]{EisensteinPoincare} we have the formula
\[
	\Ln^\dagger \rightarrow\phi_E^*\Ln^\dagger, \quad \hat{\omega}^{[k],[l]}\mapsto p^l\cdot \phi^*\hat{\omega}^{[k],[l]}.
\]
The result follows, since the Frobenius structure is by definition the inverse of this map.
\end{proof}

\begin{prop}
Let us choose some natural number $D>1$ co-prime to $N$. The analytification of the section $(1-\Phi)(\lnD)$ is given by the explicit formula:
	\begin{equation}\label{eq_phi_lnd}
		\left((1-\Phi)(\lnD)\right)^\rig\Big|_{]\tilde{t}[}=\left.\sum_{k=0}^n \left( (1+s')\frac{\partial}{\partial s'} \right)^{\circ k} \pthetaDp_{\tilde{t}}(s,s')\right|_{s'=0} \hat{\omega}^{[k,0]}\otimes\omega
	\end{equation}
	Here, $\pthetaDp_{\tilde{t}}$ is the $p$-adic theta function defined in \eqref{eq_ptheta} and $s$ and $s'$ are the coordinates corresponding to the variables $S$ and $T$ in $\VpN\llbracket S,T\rrbracket$.
\end{prop}
\begin{proof} Since $\lnD$ is obtained by restriction of $\scan^D$ to $\Etriv\times_{\Mtriv}\Inf^n_e \Edtriv$ and since $\pthetaD_{\tilde{t}}$ is the $p$-adic theta function associated to $\scan^D$ at ${\tilde{t}}$, we deduce from \cite[Lemma 9.6]{EisensteinPoincare} the formula
	\[
		\left((\lnD)\right)^\rig\Big|_{]\tilde{t}[}=\left.\sum_{k=0}^n \left( (1+s')\frac{\partial}{\partial s'} \right)^{\circ k} \pthetaD_{\tilde{t}}(s,s')\right|_{s'=0} \hat{\omega}^{[k,0]}\otimes\omega.
	\]
	Recall from \cref{prop_theta_restricted} the identity
	\[
		\pthetaDp_{(a,b)}(S,T):=\pthetaD_{(a,b)}(S,T)-\Frob^* \pthetaD_{(a,b)}([p](S),T).
	\]
	Let us denote by $\phi^*\colon \Ocal_{\Eftriv}\rightarrow \Ocal_{\Eftriv}$ the map induced by $\phi\colon \Eftriv\rightarrow\Eftriv$ on structure sheaves. \cref{lem_Frobenius_structure} shows the commutativity of the diagram
	\[
		\begin{tikzcd}
			\Gamma(\Ef,\widehat{\Lcal}_n)\ar[r,"\simeq"]\ar[d,"\Phi"] & \bigoplus_{k=0}^n \Gamma(\Ocal_{\Ef},\hat{\omega}^{[k]}\cdot \Ocal_{\Ef}) \ar[d,"\phi^*"] \\
			\Gamma(\Ef,\widehat{\Lcal}_n)\ar[r,"\simeq"] & \bigoplus_{k=0}^n \Gamma(\Ocal_{\Ef},\hat{\omega}^{[k]}\cdot \Ocal_{\Ef}).
		\end{tikzcd}
	\]
	In terms of the explicit coordinates on $\Eftriv$ coming from the trivialization $\Eftriv\righteq \Gmf{\Mtriv}$ we can describe the Frobenius lift $\phi^\#$ explicitly as
	\[
		\VpN\llbracket S \rrbracket \rightarrow \VpN\llbracket S \rrbracket, \quad f(S)\mapsto \Frob^* f([p](S)).
	\]
	Thus, we obtain
	\begin{align*}
		\left((1-\Phi)(\lnD)\right)^\rig\Big|_{]\tilde{t}[} & =\left.\sum_{k=0}^n \left( (1+s')\frac{\partial}{\partial s'} \right)^{\circ k} \left(\pthetaD_t(s,s') - \Frob^* \pthetaD_{(a,b)}([p](s),s') \right)\right|_{s'=0} \hat{\omega}^{[k,0]}\otimes\omega=\\
		&=\left.\sum_{k=0}^n \left( (1+s')\frac{\partial}{\partial s'} \right)^{\circ k} \pthetaDp_t(s,s')\right|_{s'=0} \hat{\omega}^{[k,0]}\otimes\omega.
	\end{align*}
\end{proof}

\section{The syntomic realization and moment functions of the Eisenstein measure}
In the following, we would like to give a more explicit description of the syntomic realization of the polylogarithm for elliptic curves with ordinary reduction. Let us consider the universal situation. Let $\Ee^\ord$ and $\Mm^\ord$ be the syntomic data associated to the ordinary locus of the modular curve of level $\Gamma(N)$. According to \cref{cor_synrepresentative} we already know that the syntomic polylogarithm is uniquely represented by the pair
\[
	([(\LnD,\rho_n)])_{n\geq 0}=\polsyn
\]
satisfying
\[
	\con{\Lcal}(\tilde{\rho}_n)=(1-\Phi)(\lnD)
\]
The sections $\LnD$ are constructed in a natural way out of the Poincar\'e bundle and appear in the de Rham realization of the elliptic polylogarithm. Our aim is to give a more explicit description of the overconvergent sections $\rho_n$. Let $(0,0)\neq(a,b)\in(\ZZ/N\ZZ)^2$ and $t\in E^\ord[N](M^\ord)$ resp. $\tilde{t}\in \Etriv[N](\Mtriv)$ be the associated $N$-torsion sections. Recall that we obtain an inclusion
\[
	\Gamma(]t[, (\Lcal_{n,E^\ord}^\dagger)^\rig) \hookrightarrow \Gamma(]\tilde{t}[, (\Lcal_{n,\Etriv}^\dagger)^\rig).
\]
by pullback along the canonical projection map. Let us write $\tilde{\rho}_n$ for the pullback of $\rho_n$. The infinitesimal splitting allows us to decompose $\rho_n$ as 
\[
	\tilde{\rho}_n|_{]\tilde{t}[}=\sum_{k+l\leq n}\eD\hat{\omega}^{[k,l]} 
\]
with $\eD\in\Gamma\left(B^-(0,1)\times\Mcal^\triv,\Ocal_{B^-(0,1)\times\Mcal^\triv_\Qp}^\an \right)$. It will be convenient to view $\eD$ as analytic functions on the open unit disc with values in the ring of generalized $p$-adic modular forms:
\[
	B^-(0,1)\rightarrow \VpN\otimes_{\Zp}\Qp,\quad x\mapsto \eD(x)
\]
The syntomic polylogarithm admits the following explicit description:
\begin{thm}\label{thm_mainThm}
	The elliptic polylogarithm in syntomic cohomology on the ordinary locus of the modular curve is given by the pair
	\[
		[(\rho_n,\LnD)]=\polsyn.
	\]
	Here, $(\LnD)_{n\geq 0}$ is the compatible system of $1$-forms with values in the logarithm sheaves describing the de Rham part. We have the following explicit description of $\tilde{\rho}_n$ in terms of moment functions of the $p$-adic Eisenstein measure: For $(a,b)\neq (0,0)$ let $t=t_{a,b}$ be the associated $N$-torsion section on the universal elliptic curve $E^\ord$ with $\Gamma(N)$-level structure. We have the decomposition
	\[
		\tilde{\rho}_n|_{]\tilde{t}[}=\sum_{k+l\leq n}\eD\hat{\omega}^{[k,l]} 
	\]
	with rigid analytic functions $(s\mapsto \eD(s))_{k,l\geq 0}$ on the open unit disc with values in the ring of generalized $p$-adic modular forms which are explicitly given by:
	\[
		\eD(s)=(-1)^{l} l! \int_{\Zp^\times\times \Zp} y^k x^{-(l+1)} (1+s)^x \dd \mu^{\mathrm{Eis,(p)}}_{D,(a,b)}(x,y)
	\]
\end{thm}
In the proof, we will use the differential equation
\[
	\con{\Lcal}(\tilde{\rho}_n)=(1-\Phi)(\lnD)
\]
to characterize the functions $\eD$. While this differential equation determines $\rho_n$ globally, it turns out that the local differential equation on tubular neighbourhoods of torsion sections does not have a unique solution. In \cite[Lem. 3.9]{BKT} this problem is solved by imposing a trace-zero condition making the solution unique. We follow this strategy and prove in a first step that $x\mapsto\eD(x)$ satisfies a trace-zero condition. The trace-zero condition is an immediate consequence of the distribution relation of the Kronecker section $\scan$:
\begin{lem}\label{SR_Pol_lemTrace}
	The functions $s\mapsto \eD(s)$ satisfy:
	\[
		\sum_{\zeta\in\Gmfabs[p](\Cp)}\eD\left(s+_{\Gmfabs}\zeta\right)=0,\quad \forall s\in B^-(0,1)(\Cp)
	\]
\end{lem}
\begin{proof}
	Let us write $\varphi\colon \Etriv\twoheadrightarrow E':=\Etriv/C$ for the quotient by the canonical subgroup and $\Po'$ for the Poincar\'e bundle of $E'$. The universal property of the Poincar\'e bundle gives an isomorphism
	\[
		\gamma_{\varphi,\id}\colon(\varphi\times \id)^*\Po'\righteq (\id\times\varphi^\vee)^*\Po.
	\]
	For $\tau\in \ker\varphi$ we obtain
	\begin{equation}\label{eq_phi_translation}
		(T_\tau\times \varphi^\vee)^*\Po\righteq (T_\tau\circ\varphi\times \id)^*\Po'=(\varphi\times\id)^*\Po'\righteq (\id\times \varphi^\vee)^*\Po.
	\end{equation}
	We have studied such translation operators in more generality in \cite[\S 3.3, Appendix A]{EisensteinPoincare}. To be consistent with the notation of \cite{EisensteinPoincare}, let us write
	\[
		\Ucal_{ \tau,e}^{\varphi,\id}\colon (T_\tau\times \varphi^\vee)^*\Po\righteq (\id\times \varphi^\vee)^*\Po
	\]
	for the map in \eqref{eq_phi_translation}. We have proven in \cite[Corollary A.4]{EisensteinPoincare} the following formula
	\begin{equation}\label{eq_dist}
		\sum_{\tau\in \ker\varphi(S)} ([D]\times\id)^*\Ucal_{\tau,e}^{\varphi,\id}\left( (T_\tau\times \varphi^\vee)^*\scan^D \right) =([D]\times\id)^* \gamma_{\varphi,\id}\left( (\varphi \times \id)^*(s_{\mathrm{can},E'}^D)\right).
	\end{equation}
	On the other hand, restricting \eqref{eq_phi_translation} along $E\times\Inf^n_e E^\vee$ gives the isomorphism
	\[
		\tilde{\Ucal}_\tau\colon T_{\tau}^*\Ln\righteq \Ln.
	\]
	Since $\lnD$ is obtained by restriction of $\scan^D$ along $E\times\Inf^n_e E^\vee$, we can deduce from \eqref{eq_dist} the formula
	\begin{equation*}
		\sum_{\tau\in\Ef^\triv[p]} \trans_{\tau}(T_\tau^* \lnD)= p\Phi(\lnD).
	\end{equation*}
	Since $\tau\in \ker(\varphi)$ we have $\trans_{\tau}(T_\tau^*\Phi(\lnD)))=\Phi(\lnD)$, which allows us to reformulate the above formula as
	\begin{equation}\label{eq_trace_compatibility1}
		\sum_{\tau\in\Ef^\triv[p]} \trans_{\tau}(T_\tau^* (1-\Phi)(\lnD))= 0.
	\end{equation}
	Again, by the universal property of the Poincar\'e bundle $(\Po^\dagger,\nabla_{\dagger})$ with integrable connection, we obtain an horizontal isomorphism
	\begin{equation}\label{eq_phi_translation2}
		(T_\tau\times \varphi^\dagger)^*\Po^\dagger\righteq (T_\tau\circ\varphi\times \id)^*(\Po')^\dagger=(\varphi\times\id)^*(\Po')^\dagger\righteq (\id\times \varphi^\vee)^*\Po^\dagger.
	\end{equation}
	Restricting this along $E\times\Inf^n_e E^\vee$ gives us an horizontal isomorphism
	\[
		\tilde{\Ucal}^\dagger_\tau\colon T_{\tau}^*\Ln^\dagger\righteq \Ln^\dagger.
	\]
	Both translation isomorphisms are compatible with the inclusion $\Ln\subseteq  \Ln^\dagger$, i.e. we have a commutative diagram
	\[
		\begin{tikzcd}
			T_\tau^*\Ln\ar[r,hook]\ar[d,"\tilde{U}_\tau"] & T_\tau^*\Ln^\dagger \ar[d,"\tilde{U}^\dagger_\tau"]  \\
			\Ln\ar[r,hook] & \Ln^\dagger.
		\end{tikzcd}
	\]
	In particular, we obtain from $\nabla(\rho_n)=(1-\Phi)(\lnD)$ and the horizontality of $\tilde{U}^\dagger_\tau$ the formula
	\[
		\con{\Ln^\dagger}\left(\sum_{\tau\in\Ef^\triv[p]} \trans^\dagger_{\tau}(T_\tau^*\tilde{\rho}_n)\right)=\sum_{\tau\in\Ef^\triv[p]} \trans^\dagger_{\tau}(T_\tau^*\left[(1-\Phi)\lnD\right] ).
	\]
	Using \eqref{eq_trace_compatibility1} we get
	\begin{equation}\label{SR_Pol_eq9}
		\con{\Ln^\dagger}\left(\sum_{\tau\in\Ef^\triv[p]} \trans_{\tau}(T_\tau^*\tilde{\rho}_n)\right)=0.
	\end{equation}
	But the only overconvergent section $s\in\Gamma(\bar{\Ee}_K,j_D^\dagger(\Logsyn{n}))$ satisfying the differential equation
	\[
		\con{\Ln^\dagger}\left(s\right)=0
	\]
	is  the zero section. We conclude
	\[
		\sum_{\tau\in\Ef^\triv[p]} \trans_{\tau}(T_\tau^*\tilde{\rho}_n)=0
	\]
	Recall, that $\tilde{\rho}_n|_{]\tilde{t}[}=\sum_{k+l\leq n} \hat{e}_{t,(k,l)}\hat{\omega}^{[k,l]}  $, so passing to the tubular neighbourhood $]\tilde{t}[$ proves the claimed equality
	\[
		\sum_{\zeta\in\Gmfabs[p](\Cp)}\eD\left(s+_{\Gmfabs}\zeta\right)=0,\quad \forall s\in B^-(0,1)(\Cp).
	\]
\end{proof}

\begin{proof}[Proof of \cref{thm_mainThm}]
	In the following let us consider the elliptic curve $E^\triv$, i.\,e. $\lnD$ refers to the section $l_{n,E^\triv}^D$ on $\Etriv$. Recall from \cref{prop_infsplit} that $\left((1-\Phi)(\lnD)\right)_{n\geq 0}$ is mapped to
	\[
	\left.\sum_{k=0}^n \left( (1+s')\frac{\partial}{\partial s'} \right)^{\circ k} \pthetaDp_{\tilde{t}}(s,s')\right|_{s'=0} \hat{\omega}^{[k,0]}\otimes\omega
	\]
	under
	\[
		\left(\Lcal_{n,\Etriv_K}^\dagger\right)^{rig}\Big|_{]\tilde{t}[} \righteq \bigoplus_{k+l\leq n} \hat{\omega}^{[k,l]} \Ocal^\an_{B^-(0,1)\times \Mcal^\triv_{\Qp}}.
	\]
	Thus, the differential equation
	\[
		\con{\Ln^\dagger}(\tilde{\rho}_n)=(1-\Phi)(\lnD)
	\]
	can be rewritten using the infinitesimal splitting as:
	\[
		\con{\Ln^\dagger}\left( \sum_{k+l\leq n} \eD(s) \hat{\omega}^{[k,l]} \right)=\left.\sum_{k=0}^n \left( (1+s')\frac{\partial}{\partial s'} \right)^{\circ k} \pthetaDp_{\tilde{t}}(s,s')\right|_{s'=0} \hat{\omega}^{[k,0]}\otimes \omega,\quad \forall n\geq 0
	\]
	Recall, that $\pthetaDp_t(s,s')$ is the Amice transform of the $p$-adic measure $\muEisD\Big|_{\Zp^\times\times\Zp}$. This implies the formula
	\[
		\left( (1+s')\frac{\partial}{\partial s'} \right)^{\circ k} \pthetaDp_t(s,s')\Big|_{s'=0} =\int_{\Zp^\times\times\Zp} y^k(1+s)^x d\muEisD.
	\]
	From \cref{prop_infsplit} we know that the connection $\con{\Ln^\dagger}$ expresses via the infinitesimal splitting on $\bigoplus_{k\geq 0} \Ocal_{\Gmf{\Mcal^\triv}}\omega^{[k,0]}$ as
	\[
		\con{\Ln^\dagger}(\hat{\omega}^{[k,l]})=(l+1)\hat{\omega}^{[k,l+1]}\otimes \omega.
	\] 
	Thus, we obtain the following explicit system of differential equations satisfied by $\eD(s)$:
	\begin{align*}
		(1+s)\frac{\partial}{\partial s} \hat{\mathrm{e}}_{{\tilde{t}},(k,0)} (s)&=\int_{\Zp^\times\times \Zp} y^k (1+s)^x \dd \mu^{\mathrm{Eis,(p)}}_{D,(a,b)}(x,y),\quad k\geq 0\\
		(1+s)\frac{\partial}{\partial s} \hat{\mathrm{e}}_{{\tilde{t}},(k,l)} (s)&=-l\hat{\mathrm{e}}_{t,(k,l-1)} (s),\quad l>0,k\geq 0.
	\end{align*}
	Here, we have used the fact that $\pthetaDp_{\tilde{t}}$ is the Amice transform of the measure $\mu^{\mathrm{Eis,(p)}}_{D,(a,b)}$. Further, by \cref{SR_Pol_lemTrace} the functions $\eD(s)$ satisfy the following trace-zero condition:
	\[
		\sum_{\zeta\in\Gmfabs[p](\Cp)}\eD(s+_{\Gmfabs}\zeta)=0,\quad \forall s\in B^-(0,1)(\Cp)
	\]
	\textit{Claim:} The system $\left(\eD(s)\right)_{k,l\geq 0}$ is the only system of analytic functions on $B^-(0,1)$ with values in $\VpN\otimes\Qp$ satisfying:
	\begin{enumerate}
	\item $(1+s)\frac{\partial}{\partial s} \hat{\mathrm{e}}_{{\tilde{t}},(k,0)} (s)=\int_{\Zp^\times\times \Zp} y^k (1+s)^x \dd \mu^{\mathrm{Eis,(p)}}_{D,(a,b)}(x,y),\quad k\geq 0$
	\item $(1+s)\frac{\partial}{\partial s} \hat{\mathrm{e}}_{{\tilde{t}},(k,l)} (s)=-l\hat{\mathrm{e}}_{t,(k,l-1)} (s),\quad l>0,k\geq 0$
	\item $\sum_{\zeta\in\Gmfabs[p](\Cp)}\eD(s+_{\Gmfabs}\zeta)=0,\quad \forall s\in B^-(0,1)(\Cp)$.
	\end{enumerate}
	\textit{Pf. of the claim:} The functions $\hat{\mathrm{e}}_{{\tilde{t}},(k,l)} (s)$ satisfy the above conditions. For uniqueness let $k,l\geq 0$. By induction on $k$ it is enough to show that any analytic function $F$ on $B^-(0,1)$ with values in $\VpN\otimes\Qp$ satisfying
	\begin{enumerate}
	\item[(A)] $(1+s)\frac{\partial}{\partial s}F=\begin{cases} \int_{\Zp^\times\times \Zp} y^k (1+s)^x \dd \mu^{\mathrm{Eis,(p)}}_{D,(a,b)}(x,y) & \text{if } l=0\\  -l\hat{\mathrm{e}}_{{\tilde{t}},(k,l-1)} (s) & \text{if } l>0   \end{cases}$
	\item[(B)] $\sum_{\zeta\in\Gmfabs[p](\Cp)} F(s+_{\Gmfabs}\zeta)=0,\quad \forall s\in B^-(0,1)(\Cp)$
	\end{enumerate}
	satisfies $F=\hat{\mathrm{e}}_{{\tilde{t}},(k,l)}$. Indeed, since any analytic function is given by a power series, one deduces from $(A)$ that the difference of two solutions is a constant $c\in \VpN\otimes\Qp$. By $(B)$ we conclude $p\cdot c=\sum_{\zeta\in\Gmfabs[p](\Cp)} c=0$ which implies $c=0$ and proves the claim.\par 
	Now, the theorem follows from the following observation: The sequence $(e'_{k,l})_{k,l\geq 0}$ defined by
	\[
		e'_{k,l}(s):=(-1)^{l} l! \int_{\Zp^\times\times \Zp} y^k x^{-(l+1)} (1+s)^x \dd \mu^{\mathrm{Eis,(p)}}_{D,(a,b)}(x,y)
	\]	
	satisfies:	
	\begin{enumerate}
	\item $(1+s)\frac{\partial}{\partial s} e'_{k,0}(s) =\int_{\Zp^\times\times \Zp} y^k (1+s)^x \dd \mu^{\mathrm{Eis,(p)}}_{D,(a,b)}(x,y),\quad k\geq 0$
	\item $(1+s)\frac{\partial}{\partial s} e'_{k,l}(s) =-l\cdot e'_{k,l-1}(s),\quad k\geq 0,l>0$
	\item $\sum_{\zeta\in\Gmfabs[p](\Cp)} e'_{k,l}(s+_{\Gmfabs}\zeta)=0,\quad \forall s\in B^-(0,1)(\Cp),k,l\geq 0$.
	\end{enumerate}
	Indeed, $(a)$ and $(b)$ are obvious and $(c)$ follows since $e'_{k,l}$ is the Amice transform of a $p$-adic measure which is supported on $\Zp^\times$.	From the above claim we deduce $e'_{k,l}(s)=\eD(s)$ which proves the theorem.
\end{proof}

\bibliographystyle{amsalpha} 
\bibliography{SyntomicRealization}
\end{document}